\documentclass[12pt,letterpaper]{scrartcl}
\usepackage[utf8]{inputenc}
\usepackage{tikz,ifthen}
\usepackage{amsthm,amsfonts, mathtools, amssymb, bbm,amsmath, mathdots,mathrsfs}
\usetikzlibrary{matrix,decorations.pathreplacing}
\usetikzlibrary{cd}
\usepackage[margin=1in]{geometry} 
\usepackage[square,numbers,sort&compress]{natbib}
\usepackage{graphicx}
\usepackage{transparent}
\usepackage[percent]{overpic}
\usepackage{caption}
\usepackage{subcaption}
\usepackage{xcolor,color,soul} 
\usepackage{enumitem} 
\usepackage{physics}
\usepackage[export]{adjustbox} 
\usepackage{appendix}
\definecolor{Myblue}{rgb}{0,0,0.6}  
\definecolor{Myred}{rgb}{0.6,0,0}
\definecolor{Mygreen}{rgb}{0,0.6,0}
\usepackage[colorlinks,citecolor=Myblue,linkcolor=Myblue,urlcolor=Myblue,pdfpagemode=None]{hyperref}
\allowdisplaybreaks 
\usepackage{stmaryrd}

\theoremstyle{definition}
\newtheorem{definition}{Definition}
\newtheorem{notation}[definition]{Notation}
\newtheorem{theorem}[definition]{Theorem}
\newtheorem{lemma}[definition]{Lemma}
\newtheorem{proposition}[definition]{Proposition}
\newtheorem{example}[definition]{Example}
\newtheorem{corollary}[definition]{Corollary}
\newtheorem{remark}[definition]{Remark}
\newtheorem{conjecture}{Conjecture}

\numberwithin{definition}{section}
\numberwithin{equation}{section}
\numberwithin{figure}{section}

\newcommand{\id}[1]{\operatorname{id}_{#1}} 
\newcommand{\unit}{\mathbbm{1}}
\newcommand{\Vect}{\operatorname{Vect}}
\newcommand{\GL}{\operatorname{GL}}
\newcommand{\SL}{\operatorname{SL}}
\newcommand{\prl}{\operatorname{Pr}^\mathrm{L}}
\newcommand{\Cat}{\operatorname{Cat}}
\newcommand{\ind}[1]{\widehat{#1}}
\newcommand{\Z}{\mathcal{Z}} 
\newcommand{\HC}[1]{\mathcal{HC}_{#1}}
\newcommand{\totimes}{\tilde{\otimes}}

\newcommand{\Mat}{\mathrm{Mat}}
\newcommand{\Hom}{\mathrm{Hom}} 
\newcommand{\End}{\mathrm{End}}
\newcommand{\Rep}{\mathrm{Rep}}
\newcommand{\Mod}{\mathrm{Mod}}
\newcommand{\Bimod}{\mathrm{Bimod}}
\newcommand{\Ext}{\operatorname{Ext}}
\newcommand{\Fun}{\operatorname{Fun}}
\newcommand{\GKdim}{\operatorname{GKdim}}
\newcommand{\grade}{\operatorname{j}}
\newcommand{\pdim}{\operatorname{pdim}}

\newcommand{\Oq}{\mathcal{O}_q}
\newcommand{\Dq}{\mathcal{D}_q}
\newcommand{\D}{\mathcal{D}}
\renewcommand{\O}{\mathcal{O}}
\newcommand{\OA}{\mathcal{O}_{\mathcal{A}}}
\newcommand{\OFRT}{\mathcal{O}'_{\mathcal{A}}}
\newcommand{\Ahat}{\ind{\mathcal{A}}}

\newcommand{\Aplus}{\A^+}
\newcommand{\opp}{^\mathrm{op}}
\newcommand{\OqGL}[1]{\mathcal{O}_q (\GL_{#1})}
\newcommand{\OqSL}[1]{\mathcal{O}_q (\SL_{#1})}
\newcommand{\DqGL}[1]{\mathcal{D}_q(\GL_{#1})}
\newcommand{\DqSL}[1]{\mathcal{D}_q(\SL_{#1})}
\newcommand{\Uqgl}[1]{\operatorname{U}_q \mathfrak{gl}_{#1}}
\newcommand{\Uqsl}[1]{\operatorname{U}_q \mathfrak{sl}_{#1}}
\newcommand{\DqG}{\mathcal{D}_q(G)}
\newcommand{\OqG}{\mathcal{O}_q(G)}
\newcommand{\OqFRT}{\mathcal{O}'_q (G)}
\newcommand{\OqMat}{\mathcal{O}_q (\Mat_2)}
\newcommand{\DqMat}[1]{\mathcal{D}_{q,#1} (\Mat_{2})} 
\newcommand{\DqFRT}{\mathcal{D}'_q (G)} 
\newcommand{\OqGLb}[1]{\mathcal{O}'_q (\GL_{#1})}
\newcommand{\OqSLb}[1]{\mathcal{O}'_q (\SL_{#1})}
\newcommand{\DqGLb}[1]{\mathcal{D}'_q(\GL_{#1})}
\newcommand{\DqSLb}[1]{\mathcal{D}'_q(\SL_{#1})}
\newcommand{\OqMatb}{\mathcal{O}'_q (\Mat_2)}
\newcommand{\DqMatb}[1]{\mathcal{D}'_{q,#1} (\Mat_{2})}

\newcommand{\sk}{\operatorname{Sk}}
\newcommand{\skhat}{\widehat{\operatorname{Sk}}}
\newcommand{\skmod}{\operatorname{SkMod}}
\newcommand{\skalg}{\operatorname{SkAlg}}
\newcommand{\skcat}{\operatorname{SkCat}}
\newcommand{\intskmod}{\operatorname{SkMod}^\mathrm{int}}
\newcommand{\intskalg}{\operatorname{SkAlg}^\mathrm{int}}
\newcommand{\A}{\mathscr{A}} 
\newcommand{\Ann}{\mathbb{A}\mathrm{nn}}

\newcommand{\supp}{\operatorname{Supp}}
\newcommand{\modu}{\operatorname{-Mod}}

\newcommand{\detq}{\operatorname{det}_q}
\newcommand{\numberq}[2]{\{#1\}_{#2}}

\newcommand{\ie}{i.e.\,}

\newcommand{\C}{\mathbb{C}}
\newcommand{\field}{k}
\newcommand{\Chbar}{\C\llbracket \hbar \rrbracket}
\newcommand{\Chloc}{\C(\!(\hbar)\!)}
\newcommand{\Cq}{\C[q,q^{-1}]}
\newcommand{\Cqloc}{\C(q^{\frac12})}
\newcommand{\hloc}[1]{#1^{\hbar,\mathrm{loc}}}
\newcommand{\hcmpl}[1]{#1^{\wedge}_\hbar}

\begin{document}

\title{Finiteness and holonomicity of skein modules}

\author{
	David Jordan \quad Iordanis Romaidis \\[0.5cm]
	}
\date{}

\maketitle

\begin{abstract}
We prove a generalised version of finiteness of skein modules for 3-manifolds by including boundary. We show that internal skein modules are holonomic modules over the internal skein algebra of the boundary - a property including finite generation and a Lagrangian support condition. Our approach involves defining skein transfer bimodules, topologically obtained by 2-handle attachments, and establishing $q$-analogues of statements in D-module theory to prove preservation of holonomicity. 
\end{abstract}

\setcounter{tocdepth}{2}
\tableofcontents

\section{Introduction}
This paper concerns the fine algebraic structure of skein modules of oriented 3-manifolds -- namely the properties of finite generation and holonomicity of the skein module of an oriented 3-manifolds relative to the skein algebra of its boundary.

The skein module  $\skmod_G(M)$ is a vector space assigned to an oriented 3-manifold $M$ and a reductive group $G$, which mathematically captures the space of Wilson loop operators in Chern-Simons theory or, equivalently, certain sectors of line operators constrained to a 3-dimensional Dirichlet boundary in the Kapustin-Witten twist of 4D $\mathcal{N}=4$ super Yang-Mills \cite{Witten, Kapustin-Witten}.  

Much aside from their physical origins, skein modules are interesting to mathematicians because they generalise the construction of Reshetikhin-Turaev knot invariants \cite{RT} such as the Jones polynomial to (vector-valued) invariants of knots in arbitrary 3-manifolds.  They are expected to arise via $\beta$-factorisation homology (see \cite{AFRbeta}) as a generalised homology theory for 3-manifolds, allowing non-abelian coefficients.  They are interesting geometrically because of their classical degeneration to character varieties -- moduli spaces of representations of $\pi_1(M)$, or equivalently of $G$-local systems.  More recently, they have been related (largely conjecturally) to cohomological Donaldson--Thomas invariants of character stacks \cite{GS} (see also \cite{Kaubrys,Bu+}).  Altogether, they are a fundamental object of study in the field of quantum topology.

The definition of the skein module is elementary: a $G$-skein consists of an embedded ribbon graph in $M$ with edges and vertices labelled by representations and morphisms of $G$-modules, or more precisely representations of the quantum group $U_q\mathfrak{g}$.  The $G$-skein module $\skmod_G(M)$ is the linear span of such labelled graphs, modulo local relations which encode the graphical calculus for representations of the quantum group.

\subsection{Results} 

The paper \cite{GJS} established that for reductive groups $G$, for \emph{closed} 3-manifolds, and for \emph{generic} values of the quantum parameter $A=q^{\frac12}$, the skein module $\skmod_G(M)$ is finite-dimensional.  In this paper, we establish the analogous finiteness conjecture for skein modules of 3-manifolds with arbitrary boundary, for the groups $G=\SL_2(\C),\GL_2(\C)$.  Namely, we prove:

\begin{theorem}\label{thm:Detcherry-conjecture} Let $G=\mathbb{C}^\times$, $\SL_2$, or $\GL_2$.  Suppose that $q$ is generic and that $M$ has a (possibly empty) boundary $\partial M = -\Sigma_\mathrm{in}\cup\Sigma_\mathrm{out}$.  Then $\skmod_G(M)$ is finitely generated over the algebra $\skalg_G(\Sigma_\mathrm{in})\otimes \skalg_G(\Sigma_\mathrm{out})^{op}$.
\end{theorem}

This statement has appeared as Conjecture 3.1 of \cite{Det19}. For $G=\SL_2$, Theorem~\ref{thm:Detcherry-conjecture} was known for some families of 3-manifolds \cite{LT,Mar,AF,KaruoWang}, see \cite{Belletti-Detcherry} for more extensive discussion of these examples. Theorem \ref{thm:Detcherry-conjecture} is best understood as an application of our main result which is much stronger and more precise:

\begin{theorem}[Proved as Theorem \ref{thm:main-thm-skein}]\label{thm:Detcherry-conjecture-holonomic} Let $G=\mathbb{C}^\times$, $\SL_2$, or $\GL_2$.  Suppose that $q$ is generic, and that $M$ has a boundary $\partial M = -\Sigma_\mathrm{in}\cup\Sigma_\mathrm{out}$.  Then $\intskmod_G(M^*)$ is finitely generated and holonomic as a $\intskalg_G(\Sigma^*_\mathrm{in})$-$\intskalg_G(\Sigma^*_\mathrm{out})$-bimodule.
\end{theorem}

Two notions in the above theorem require unpacking -- that of \emph{internal} skein modules $\intskmod_G(M^*)$ and skein algebras $\intskalg_G(\Sigma^*)$, and the general property of a module over a deformation quantization algebra being \emph{holonomic}.  We now explain these notions in turn.

\paragraph{Internal skein modules} Recall that even when $M$ has boundary, $\skmod_G(M)$ is typically defined in such a way that skeins do not approach the boundary.  By contrast, the \emph{internal skein module} $\intskmod_G(M^*)$ of a 3-manifold -- as well as the closely related internal skein algebra $\intskalg_G(\Sigma^*)$ of a surface $\Sigma$ -- is defined so that skeins can exit $M$ (resp. $\Sigma\times I$) through a prescribed collection of disks (say $k\geq 0$ many) contained in the boundary (called ``gates"). For simplicity we will abbreviate by $M^*$ and $\Sigma^*$, respectively, the most typical configuration in which we have exactly one gate in each boundary component of $M$, respectively, connected component of $\Sigma$.

This leads to the definition of the internal skein module $\intskmod_G(M^*)$ and similarly the internal skein algebra $\intskalg_G(\Sigma^*)$ as an object of $\Rep_q(G)^{\boxtimes k}$, whose $(V_1\boxtimes \cdots \boxtimes V_k)$-typical component measures skeins ending at the prescribed disks and coloured with $V_1.\ldots, V_k$, respectively.  Internal skein modules/algebras are functorial avatars of so-called ``stated" skein modules/algebras introduced independently in \cite{Costantino-Le}; the two notions were shown to be isomorphic in \cite{Haioun}.  Presentations of internal skein algebras in the presence of a single gate are computed in \cite{BZBJ18a,Cooke-thesis}; while internal skein algebras with multiple gates were considered in \cite{Haioun}, no presentation was given there. In Section~\ref{sec:int-sk}, we employ the monadic techniques from \cite{BZBJ18a}, to give an explicit presentation of the internal skein algebra of a surface with multiple gates.

\paragraph{Holonomicity} Recall that whenever some non-commutative algebra $A_q$ flatly degenerates to a commutative algebra at $q=1$, the degeneration inherits a Poisson bracket, via the formula $\{f,g\} = \lim_{q\to 1} \frac{\widetilde{f}\widetilde{g}-\widetilde{g}\widetilde{f}}{q-1}$, where $\widetilde{f}$, $\widetilde{g}$ are arbitrary choices of lifts to $A_q$ of some functions $f,g\in A_{q=1}$.  Let us suppose now that $\mathscr{X} = \mathrm{Spec}(A_{q=1})$ is smooth, and that the induced Poisson bracket is in fact symplectic.  Given a finitely generated non-zero module $M_q$ over $A_q$, we can form its specialisation at $q=1$.  The set-theoretic support of the underlying coherent sheaf on $\mathscr{X}$ is called the singular support of $M_q$. The specialisation -- and in particular the singular support -- of $M_q$ depend on the choice of generators; however the dimension of the singular support is independent of the choice of generators.  This dimension is bounded from below by half the dimension of $\mathscr{X}$, and we say that $M_q$ is holonomic when this bound is obtained for some choice of generators.

Hence, not only is $\intskmod_G(M)$ finitely generated relative to its boundary, it is as small as possible in a precise sense: it is ``maximally over-determined" (a synonym for holonomic) as a module over the internal skein algebras of its boundary, consequently it has the smallest possible homological dimension and the smallest possible Gelfand-Kirillov dimension (see Section \ref{subsec:def-quant-hol-mods}).  
We return to potential applications of holonomicity in Section \ref{subsec:future}.

\paragraph{Related work}
While writing up this work, we learned through private correspondence with Renaud Detcherry of the results in \cite{Belletti-Detcherry}, which establish a different, complementary finiteness property for $G=\SL_2$ -- also conjectured by Detcherry \cite{Det19} -- namely finite-dimensionality over a field of rational functions of a maximal commutative subalgebra of the skein algebra of the boundary surface of a closed 3-manifold.  For manifolds with boundary, the two finiteness results are complementary and neither appears to imply the other.

Like \cite{Belletti-Detcherry}, our results establish that the peripheral ideal of any link is non-empty (as conjectured in \cite{Frohman-Gelca-Lofaro}, and give a new proof of the holonomicity of the coloured Jones polynomial -- famously proved in \cite{Garoufalidis-Le}.  These are both special cases where $M$ is a knot complement, and our main results generalise and prove these statements to oriented 3-manifolds with arbitrary boundary.

\subsection{Methodology}
Let us now outline our methodology, which differs significantly in key respects from the approach of \cite{GJS}.  These differences explain both why our results are much stronger than in \cite{GJS}, but also why (for the time being) they are limited to $\C^\times, \SL_2$, and $\GL_2$.  Essentially, in order to deploy the method in this paper one needs to understand the structure of internal skein algebras in much finer detail.

Whereas \cite{GJS} relied on abstract deformation theory arguments \`a la Kashiwara and Schapira \cite{KS} to ultimately reduce to a question of $\D(G)$-modules and the theory of ordinary differential operators, in the present work we instead follow the multiplicative framework of $q$-difference operators, or $\Dq$-modules, \`a la Sabbah \cite{Sabbah}.

\paragraph{Compression bodies and Heegaard splittings}
The first key ingredient in our construction is topological in nature, and classical:  whereas a closed connected 3-manifold always admits a Heegaard splitting into the union of two handlebodies of genus $g$ glued along their common boundary, an arbitrary compact connected 3-manifold with boundary admits a generalised Heegaard splitting into a genus $g$ handlebody and a compression body containing the boundary (see \cite{Hempel} for a comprehensive introduction).  A compression body is a 3-manifold obtained from a thickened surface $\Sigma_g\times I$ by adding in 2-handles to collapse cycles in $\Sigma_g$, and then 3-handles to cap off any resulting $S^2$ boundaries.  Alternatively, a compression body can be obtained from a genus $g$ handlebody by removing embedded handlebodies of smaller genus.  Together with twisting by the mapping class group, the 2-handles to be glued come in two types: they are either separating -- each such gluing adds a connected component to the boundary -- or they may be non-separating, in which case they may be assumed to collapse a standard $A$-cycle on the surface (see Section \ref{sec:transfer bimodules}).

\paragraph{Internal skein transfer bimodules}
Let $\Sigma^*_{g,r}$ denote the oriented surface of genus $g$, with $r$ boundary components, each of which is equipped with a gate.  In Section \ref{sec:int-sk} we give an isomorphism
\[\intskalg_G(\Sigma^*_{g,r}) \cong \D_q(G)^{\totimes g}\widetilde{\otimes} \D'_q(G)^{\totimes(r-1)}.\]
Here, $\D_q(G)$ and $\D'_q(G)$ denote two closely related algebras of $q$-difference operators on $G$ (see Section \ref{sec:int-sk} for detailed definitions), and $\widetilde{\otimes}$ and denotes a twisted tensor product of algebras determined from a ``gluing pattern" for $\Sigma_{g,r}^*$ (see Section \ref{sec:int-sk}).  With this presentation the algebra $\intskalg_G(\Sigma^*_{g,r})$ quantizes the Fock-Rosly Poisson structure on $G^{2(g+r-1)}$ \cite{Fock-Rosly}, and is a mild generalisation of the Alekseev-Grosse-Schomerus algebras \cite{AGS}.

The algebras $\D_q(G)$ and $\D'_q(G)$ are each themselves smash products,
\[\D_q(G) = \OqG\sharp \OqG,\qquad \D'_q(G) = \OqG\sharp \OqFRT,\]
where $\OqG$ and $\Oq'(G)$ denote the reflection equation (RE) algebra and the Fadeev-Reshetikhin-Takhtajan (FRT) algebras, respectively.  Accordingly, each has a canonical induced module,
\[
\OqG = (A-I)\backslash\DqG ~ \circlearrowleft~\DqG, \qquad \Oq'(G) = (A-I)\backslash\D'_q(G)~\circlearrowleft ~\Dq'(G),
\]
where in both cases we abbreviate by $(A-I)$ the right ideal obtained as the kernel of the counit homomorphism $\epsilon:\OqG\to\field$.

When attaching a 2-handle along a standard separating or non-separating cycle $\alpha$ in $\Sigma_{g,r}$, we are essentially setting the corresponding $G$-factor to be trivial.  Topologically, this is implemented by the internal skein transfer bimodule of the compression body $C_\alpha$ obtained from $\Sigma_g\times [0,1]$ by attaching a 2-handle along $\alpha\times \{0\}$.

In the non-separating case, this corresponds to the the $\intskalg_{q,G}(\Sigma^{\ast}_{g-1,r})\mbox{-}\intskalg_{q,G}(\Sigma_{g,r}^{\ast})$-bimodule
\[\intskmod_{q,G}(C_\alpha) \cong \OqG \totimes \DqG^{\totimes g-1} \totimes \Dq'(G)^{\totimes r-1}~.\]

In the separating case, we instead obtain the $\intskalg_{q,G}(\Sigma_{g_1,r_1}^\ast\cup \Sigma_{g_2,r_2}^\ast)\mbox{-}\intskalg_{q,G}(\Sigma_{g,r}^{\ast})$-bimodule
\[\intskmod_{q,G}(C_\alpha) \cong \D_q(G)^{\totimes g_1}\totimes \D'_q(G)^{\totimes r_1-1}\totimes \OqFRT \totimes  \D_q(G)^{\totimes g_2}\totimes \D'_q(G)^{\totimes r_2-1} ~,\]
where $g_1,g_2$ and $r_1,r_2$, respectively are the genus and number of boundary components in the two components separated by $\alpha$, hence satisfying $g=k_1+k_2$, $r = r_1 + r_2$. 
The details of this construction and bimodule action will spelled out in Section~\ref{sec:transfer bimodules}.

\paragraph{Direct and inverse image for skein modules}

Perhaps the most significant insight of this paper is that the internal skein transfer bimodule is a precise $q$-analogue of the transfer bimodule which features in the theory of algebraic differential equations, \ie the study of modules over the algebra $\D$ of polynomial differential operators on a smooth algebraic variety.

In the classical theory of $\D$-modules, a closed immersion $i: Z\to X$, and its complementary open embedding $j:X\setminus Z\to X$, induce functors
\[
i_*:\D(Z)\modu \leftrightarrows \D(X)\modu:i^\ast,\qquad
j^*:\D(X)\modu\leftrightarrows \D(X \setminus Z)\modu:j_\ast~.
\]
For $M\in \D(Y)$ we have an exact triangle,
\begin{equation}\label{eqn:exact-triangle}
i_\ast Li^\ast M[n] \rightarrow M \rightarrow Rj_\ast j^\ast M \xrightarrow{+1}
\end{equation}
where $n$ is the codimension of $Z$ in $X$ appearing as a shift. 
Moreover, the functors $i_\ast$ and $i^\ast$ is computed via an explicit ``transfer bimodule", $\D_{X\to Y}$.

The algebra $\intskalg(\Sigma^*_{g,r})$ is a $q$-analogue of the algebra $\D(G^{g+r-1})$ of polynomial differential operators on the smooth algebraic variety $G^{g+r-1}$; in fact upon appropriate completion, $\intskalg(\Sigma^*_{g,r})$ degenerates to $\D(G^{g+r-1})$.  This fact was observed and used heavily in \cite{GJS} when $r=1$, however the relation to differential operators for $r\geq 2$ is novel to this paper, and plays an essential role for manifolds with multiple boundary components.

Using this analogy, we define by hand the functors $i^*$,$i_*$,$j^*$,$j_*$, via internal skein transfer bimodules for handle attaching.  Hence, viewed through the lens of our analogy with the theory of $\D$-modules, we obtain an algebraic description of the internal skein module of a compression body as a composition of ``$q$-transfer bimodules for closed immersions" $G^k\to G^g$, for $k<g$.

\paragraph{Preservation of holonomicity}  The final key ingredient in our construction is analytic in nature, and perhaps the most important result in the field of geometric representation theory: tensoring with transfer bimodules preserves holonomicity.  The proof of this classical result most relevant to the current paper is due to Bernstein \cite{Bernstein} (see also \cite{Bjork}).  He considers the algebra $\D(\mathbb{A}^n)$, governing partial differential equations in $n$ variables.  He performs an induction on $n$, starting from the most fundamental fact in the theory of differential equations, namely that ODE's in a single variable admit a finite-dimensional space of solutions, and appealing to the exact triangle to perform the inductive step.

Building on Bernstein's technique, Sabbah established a number of fundamental parallel results for the parallel but more complicated algebra $\Dq((\mathbb{G}_m)^n)$, of $q$-difference operators on the affine torus of rank $n$.  Namely, he constructed the functors $(i_*,i^\ast)$, their description via transfer bimodules, the functors $(j_\ast,j^*)$, and an analogue of the exact triangle, all in the $q$-setting.  Using these tools he showed that transfer bimodules for closed immersions between tori preserve holonomicity.  An added difficulty he overcame is that in defining $j^*$, one cannot simply delete the image $i(Z)$, but one must also delete it's arbitrary $q$-shift; otherwise the Ore condition for localisation is not satisfied.

We adopt a similar approach to Bernstein and Sabbah, to show that the $q$-transfer bimodules associated to 2-handle attachments preserve holonomicity, by treating the separating and non-separating cases in parallel (see Theorem \ \ref{thm:preservation}).  Namely we establish an analogue of the exact triangle \eqref{eqn:exact-triangle} for internal skein transfer bimodules, and we also prove an analogue of Kashiwara's theorem for these, see Proposition \ref{prop:kernel}.

In fact, the case $G=\C^\times$ of Theorem \ref{thm:Detcherry-conjecture-holonomic} follows from Sabbah's results which apply directly there. Several challenges arise however when consider non-abelian groups $\SL_2(\C)$ and $\GL_2(\C)$.  Whereas Sabbah's basic building block is the algebra $\D_q(\C^\times)$ of $q$-difference operators on the multiplicative group $\C^\times$, our basic building block is instead the algebra $\DqG$ of $q$-difference operators on the reductive groups $\GL_2$ and $\SL_2$. Consequently, whereas $\D_q(\C^\times)$ has a commutative base $\C^\times$ accessible via the Lagrangian subalgebra $\O(\C^\times)\subset \Dq(\C^\times)$ and amenable to geometric definitions, by contrast our ``base" subalgebra $\OqG$ is itself a quantum object, and is itself non-commutative as soon as $G$ is non-abelian reductive.

A further challenge we overcome is that our induction step via internal skein transfer bimodules passes from $G^{k-1}$ to $G^{k}$ via a closed immersion $G^{k-1}\to G^{k}$ (more precisely our setup is a deformation quantization of that). Even classically, one has the complication that the complementary embedding $j:(G^k\setminus G^{k-1})\to G^k$ is not affine. In order to construct the necessary exact triangle to perform the induction step, we must therefore construct a non-commutative affine open chart on $\OqG$, consisting of four affine opens whose union covers the complement $G\setminus\{e\}$ to the identity element $e\in G$. We establish the Ore property and finite generation of each localisation individually, and finally we apply an analogue of the exact triangle and of Kashiwara's theorem to conclude that the transfer bimodule preserves holonomicity.

\subsection{Further directions}\label{subsec:future}

Our results open up several avenues for future research and many new questions, some of which we pose here.

\paragraph{Extension to reductive groups}
While Theorem \ref{thm:Detcherry-conjecture-holonomic} applies only to the cases $G=\C^\times$,$ \GL_2$ and $\SL_2$, there is no essential impediment to applying the strategy here more generally, and indeed many of the theorems in the paper are for a general reductive group $G$. The main remaining challenge in our approach is to construct the affine cover of $G\setminus e$ algebraically in $\OqG$, and undertake to prove the finite generation of each affine chart, as is done in Section~\ref{subsec:Ore-Loc}.

This is well within reach for all groups of Type $A$, and perhaps for matrix groups more generally using the direct approach in the paper. An alternative approach suggested by Pavel Etingof would be to appeal to the techniques of \cite{Etingof} involving $p$-supports for cluster varieties, which should encompass character varieties and character stacks. This is a direction we intend to pursue in future work.

\paragraph{Skein modules with defects}
Theorem \ref{thm:Detcherry-conjecture-holonomic} has the following immediate consequence.
\begin{corollary}
Let $G=\GL_2$, $\SL_2$ or $\C^\times$, suppose that $\partial M=T^2$, and let $Y$ be an arbitrary holonomic $\DqG$-module  Then the relative tensor product,
\[ \intskmod(M)\otimes_{\DqG} Y\]
is finite-dimensional.
\end{corollary}

Natural examples of $Y$ include $\OqG$ and its twists by $\SL_2(\mathbb{Z})$ -- taking these will just recover the finite-dimensionality of the resulting surgered 3-manifold.  However, we may obtain other examples of holonomic modules $Y$, for instance coming from quantum symmetric pairs, and from parabolic induction and restriction: in this way we obtain integer-valued invariants (the dimension) of a closed 3-manifold, marked with a knot or link inside it. In this way, one can investigate a topological Langlands duality for skein modules with such line defects (see \cite{BZSV} and \cite{JLanglands}). We intend to return to the study of such examples in future work with Eric Chen.

\paragraph{Derived skein modules}

The reader will have noticed the appearance of exact triangles in the methodology section.  In fact our approach establishes not only holonomicity of skein modules for 3-manifolds with boundaries, it points towards a derived enhancement of this statement.  Namely, taking everywhere the left-derived tensor product with $q$-transfer bimodules gives a candidate for the derived internal skein module of the compression body, and hence (by again taking derived tensor products with the internal skein algebra of the Heegaard splitting) a candidate for the derived internal skein module of an arbitrary 3-manifold with boundary.  Our proof of finiteness holds verbatim to establish the finite-dimensionality of derived skein modules constructed this way in all homological degrees.

\paragraph{Extension to $q$ not a root of unity}
In contrast with skein algebras of surfaces (which are free in the quantum parameter $q$), skein modules of 3-manifolds have plentiful $q$-torsion.

The methods of \cite{GJS} following \cite{KS} only allowed us to access $q$-torsion at $q=1$, however our approach in this paper following Sabbah \cite{Sabbah} has opened the possibility of specialising $q$ to arbitrary complex values.

\begin{conjecture}\label{conj:not-a-root-of-one}
The assertions of Theorems \ref{thm:Detcherry-conjecture} and \ref{thm:Detcherry-conjecture-holonomic} hold whenever $q\in\C^\times$ is not a root of unity.
\end{conjecture}

The main missing step towards proving Conjecture \ref{conj:not-a-root-of-one} is the following:

\begin{conjecture}
Suppose that $q$ is not a root of unity, and let $G=\GL_2$ or $\SL_2$.  Then Bernstein's inequality holds for $\DqG$, \ie the Gelfand-Kirillov dimension of any simple, strongly equivariant $\DqG$-module is at least $\dim G$.
\end{conjecture}

While establishing Bernstein's inequality is a difficult problem in general, it may be tractable in the specific cases $G=\GL_2$, $\SL_2$.  We hope to return to these questions in a future paper.

\paragraph{Parabolic restriction and quantum $A$-polynomial}  Another example of a defect skein module is the thickened surface $\Sigma_g\times I$, with a bipartite colouring into $G$ in the region $t>\frac12$, and $T$ in the region $t<\frac12$, with $\Sigma_g\times \{\frac12\}$ coloured with $B$.  Recently the papers \cite{JLSS} and \cite{BrownJ} introduced the notion of decorated and redecorated quantum character stacks of 3-manifolds as a way to study such defects.  We expect our methods in this paper can be adapted to show that the functors induced by these defects also preserve holonomicity. 

\subsection{Outline}

The remainder of the paper is organised as follows.  In Section \ref{sec:background}, we recall several preliminaries.  Section \ref{subsec:cats} recounts the notions of $\field$-enriched categories, bimodules, co-ends and related constructions.  Section \ref{subsec:bfa-moment} recalls the braided function algebra, the quantum Harish-Chandra category, moment maps, and the central notion of strong equivariance for algebras.  Section \ref{subsec:skein-mods} recalls the basic facts we require from skein theory, and finally Section \ref{subsec:def-quant-hol-mods} recalls a number of basic facts we will require from the theory of deformation quantization.

In Section \ref{sec:int-sk} we consider internal skein algebras and modules in the case of multiple gates, and we give a monadic reconstruction echoing that of \cite{BZBJ18a,BZBJ18b}.  A new ingredient is the appearance of the algebra $\Dq'(G)$, which gets an interpretation as an internal skein algebra for an annulus with gates on each boundary component.

In Section \ref{sec:transfer bimodules} we introduce skein transfer bimodules, and in particular we focus on two types of such bimodules: those for 2-handle attachments along separating and non-separating curves.  We show that in both cases the resulting transfer bimodule is an induced module from a trivial character for a flat subalgebra $\OqG$, sitting inside either a $\Dq(G)$ or $\Dq'(G)$ factor, respectively, in the internal skein algebra.

In Section \ref{sec:MainThm} we state and prove the main theorem, asserting the finiteness and holonomicity of the skein module relative to its boundary. In Section \ref{subsec:GL2-background}, we recall in more detail the algebras of $q$-difference operators on $\GL_2$ and $\SL_2$. In Section \ref{subsec:GL2-Koszul}, we recall Koszul resolutions for the counit of the braided function algebra.  In Section \ref{subsec:Ore-Loc} we define the Ore sets which quantize an open chart complementary to the identity element in $G$, and moreover we establish that the localisation functor to each chart preserves finiteness and holonomicity.  In Section \ref{subsec:les} we construct the $q$-analogue of the usual long exact sequence relating derived inverse images to localisations.  In this way we show that the derived inverse image preserves holonomicity.

\subsection{Acknowledgements}
Several years ago, we were encouraged by Charlie Frohman, and independently by Theo Johnson-Freyd, to pursue finite generation of skein modules of 3-manifolds with boundary, via generalised Heegaard decompositions.  While it took many years for these ideas to percolate to their present form, the vision of these two friends is very much present in this work, and we are grateful for their generosity and encouragement.

We have also benefited immensely from discussions with David Ben-Zvi, Sam Gunningham and Pavel Safronov regarding six-functor formalism for $\D$-modules and its potential relations to skein theory, and with Jennifer Brown, Ben Ha\"ioun and Jackson Van Dyke regarding many algebraic aspects of skein theory. We are also grateful to Pavel Etingof for enlightening conversations about potential generalisations of our results to arbitrary groups $G$, and to all $q$ not a root of unity.  Finally, we are grateful to Giulio Beletti and Renaud Detcherry for updates about their related works and well as many insightful conversations.

The work of both authors was supported by the Simons
Foundation award 888988 as part of the Simons Collaboration on Global Categorical Symmetry.  The work of DJ was supported by the EPSRC Open Fellowship ``Complex Quantum Topology", grant number EP/Y008812/1. 

\section{Background}
\label{sec:background}
\noindent\textbf{Conventions:}
\begin{itemize}
    \item We say that $q$ is \textbf{generic} to indicate we are working over $\Cqloc$. We say that $q$ is \textbf{not a root of unity} to mean that either $q$ is generic or specialised at a number $q\in \C^\times$ such that $q^l \neq 1$ for all $l\in \mathbb{N}$.
    \item For an algebra $A$ over a commutative ring $k$ and $\lambda \in k$, we denote the $\lambda$-\textbf{commutator} of two elements $a,b\in A$ by $[a,b]_\lambda:= ab-\lambda ba$. 
    Similarly, the $\lambda$-\textbf{anticommutator} will be denoted by $\{a,b\}_\lambda := [a,b]_{-\lambda}= ab + \lambda ba.$
    \item For $n\in\mathbb{Z}$ and $q\in \field^\times$, we write $\numberq{n}{q}:= (q^n-1).$
\end{itemize}

\subsection{Categories}\label{subsec:cats}
We recall some categorical notions to fix notation and set up the algebraic target of skein theory as a categorified TQFT, see \cite{Cooke-thesis} or \cite{GJS} for details. 
\begin{definition}
The 2-category $\Cat$ of $\field$-linear categories consists of: 
\begin{itemize}
    \item Objects: Small $\field $-linear categories, 
    \item $\Hom$-category: The category of morphisms is given by the category of $\field$-linear functors: 
    \begin{equation*}
        \Hom_{\Cat}(\mathcal{C},\mathcal{D}) := \Fun_{\field}(\mathcal{C},\mathcal{D})~.
    \end{equation*}
\end{itemize}
\end{definition}
Given two $\field$-linear categories $\mathcal{C}$ and $\mathcal{D}$, their $\field$-linear tensor product will be denoted by $\mathcal{C}\boxtimes \mathcal{D}$. This endows $\Cat$ with a symmetric monoidal structure with the monoidal unit being the category of $\field$-modules $\Mod_k$. 

\begin{definition}
    The bicategory $\Bimod$ of $\field$-linear categories and bimodules consists of: 
    \begin{itemize}
        \item Objects: Small $\field$-linear categories, 
        \item Hom-category: $\Hom_{\Bimod}(\mathcal{C},\mathcal{D}) = \Fun_{\field}(\mathcal{C}\boxtimes \mathcal{D}\opp, \Mod_{\field})$~.
    \end{itemize}
\end{definition}
Composition of $F: \mathcal{C} \boxtimes \mathcal{D}\opp \rightarrow \Mod_\field$ with $G: \mathcal{D}\boxtimes \mathcal{E}\opp \rightarrow \Mod_\field$ is given by the coend \cite[Ch.\ 7.8]{Bor94}:
\begin{equation*}
    G\circ F := \int^{d\in \mathcal{D}}{F(-, d) \otimes G(d, -)}: \mathcal{C}\boxtimes E\opp \rightarrow \Mod_\field~. 
\end{equation*}
The $\field$-linear tensor product $\boxtimes$ equips $\Bimod$ with a symmetric monoidal structure \cite[Sec.\ 7]{DS97}.

The bicategory $\Bimod$ can be seen as an enlargement of $\Cat$ at the level of morphisms in the following sense: There is a symmetric monoidal fully faithful functor $\Cat\rightarrow \Bimod$ which is the identity on objects and sends a functor $F: \mathcal{C}\rightarrow \mathcal{D}$ to the bimodule \[\Hom_{\mathcal{D}}\left(-, F(-)\right): \mathcal{C}\boxtimes \mathcal{D}\opp\rightarrow \Mod_\field~.\] 

The categories that are considered in this text will typically not be closed under colimits and thus we often pass to their cocompletions which will live in the category of locally presentable categories, see \cite{BCJF} for details.
\begin{samepage}
\begin{definition}
    The bicategory of locally presentable categories $\prl$ consists of: 
    \begin{itemize}
        \item Objects: Locally presentable $\field$-linear categories, 
        \item Hom-category: $\Hom_{\prl}(\mathcal{C},\mathcal{D}):= \Fun_{cc}(\mathcal{C},\mathcal{D})$ (cocontinuous functors)~.
    \end{itemize}
\end{definition}
\end{samepage}

The Deligne-Kelly tensor product equips $\prl$ with a symmetric monoidal structure with monoidal unit $\Mod_\field$.

There is a symmetric monoidal fully faithful functor $\widehat{(\mbox{-})}:\Bimod \rightarrow \prl$ defined by
\begin{equation}\label{def:cocompletion}
    \mathcal{C} \mapsto \widehat{\mathcal{C}} := \Fun_{\field}(\mathcal{C}\opp, \Mod_\field) 
\end{equation}
and 
\begin{equation*}
    F: \mathcal{C}\boxtimes \mathcal{D}\opp \rightarrow \Mod_\field \mapsto \left(\widehat{F}: \widehat{\mathcal{C}} \rightarrow \widehat{\mathcal{D}},~ X \mapsto \int^{c\in \mathcal{C}}{F(c,-)\otimes X(c)}\right)~.
\end{equation*}
In other words, $\widehat{\mathcal{C}}$ is the category of $\Mod_\field$-valued presheaves on $\mathcal{C}$. It is realized as a free cocompletion of $\mathcal{C}$ via the Yoneda embedding $\mathcal{C}\rightarrow \widehat{\mathcal{C}}, c \mapsto \Hom_{\mathcal{C}}(-,c)$. 

The cocompletion functor $\widehat{(\mbox{-})}: \Cat \rightarrow \prl$ carries algebras in $\Cat$ to algebras in $\prl$. In particular, if $(\mathcal{C},\otimes_{\mathcal{C}},\unit_\mathcal{C})$ is a monoidal category in $\Cat$, its cocompletion $\ind{\mathcal{C}}$ is equipped with a monoidal structure via the \textit{Day convolution} \cite{Day}: For $X,Y \in \ind{\mathcal{C}}$, their tensor product is given by 
\begin{equation*}
    X\otimes_{\ind{\mathcal{C}}}Y := \int^{c,c' \in \mathcal{C}}{\Hom_{\mathcal{C}}(-, c\otimes_{\mathcal{C}} c')\otimes X(c)\otimes Y(c')}~
\end{equation*}
with monoidal unit $\Hom_{\mathcal{C}}(-,\unit_{\mathcal{C}})$. Moreover, if $\mathcal{C}$ is braided, then $\ind{\mathcal{C}}$ becomes braided. 

Note that an algebra structure in $\ind{\mathcal{C}}$ is the same as the structure of a lax monoidal functor. In other words, an algebra in $\ind{\mathcal{C}}$ is a presheaf $X: \mathcal{C}\opp \rightarrow \Mod_\field$ along with natural transformations
\[
\varphi_{c,c'}: X(c)\otimes X(c') \rightarrow X(c\otimes_\mathcal{C} c')
\]
and 
\[
\eta:\field \rightarrow X(\unit_\mathcal{C})~.
\]
compatible with associators and unitators. A module $Y\in \ind{\mathcal{C}}$ over such an algebra $X$ comes with natural transformations
\[
X(c) \otimes Y(c') \rightarrow Y(c\otimes c')
\]
compatible with the lax monoidal structure of $X$. 

\subsection{Braided function algebra and moment maps}\label{subsec:bfa-moment}

Let $\mathcal{A}$ be a $\field$-linear ribbon category and let $T: \mathcal{A}\boxtimes \mathcal{A} \rightarrow \mathcal{A}$ denote its tensor product functor. This admits a right adjoint $T^R$ when passing to the free cocompletion, \ie we have the following adjunction 
\begin{equation*}
    T: \Ahat\boxtimes\Ahat \overset{\dashv}{\leftrightarrows}\Ahat: T^R
\end{equation*}
The braiding in $\mathcal{A}$ turns $T$ into a (strong) monoidal functor and $T^R$ into a lax monoidal functor. The following definitions have been constructed in \cite{Majid,Lyub}.

\begin{definition}
The braided function algebra $\OA$ is
\[
    \OA=T\circ T^R(\unit) = \int^{x \in \mathcal{A}}{x^{\ast} \otimes x} \in \Ahat 
\]
\end{definition}
The braided function algebra carries a Hopf algebra structure: The algebra structure is induced from the (lax) monoidality of $T$ and $T^{R}$ and the coalgebra structure is induced by the adjunction. Finally, the antipode $S:\OA \rightarrow \OA$ uses the ribbon structure of $\mathcal{A}$. 

\begin{example}
    Let $q$ be not a root of unity and consider $\mathcal{A} = \Rep^\mathrm{fd}_qG \subset \Ahat = \Rep_q G$. The braided function algebra obtains the form 
    \begin{equation*}
        \OA \cong \bigoplus_{\lambda \in \Lambda^\mathrm{dom}}{V(\lambda)^\ast \otimes V(\lambda)}
    \end{equation*}
by the Peter-Weyl theorem where $\Lambda^{\mathrm{dom}}$ denotes the set of dominant weights. It coincides with the so-called \textit{reflection equation (RE) algebra} $\OqG$, which we describe in detail for $G=\GL_2, \SL_2$ in Section~\ref{subsec:GL2-background}.

    We also note that the algebra $\OFRT:=T^R(\unit) = \int^{x} x^\ast \boxtimes x \in \Ahat^{\boxtimes 2}$ in this example coincides with the FRT algebra \cite{FRT} which we denote by $\OqFRT$ to distinguish it from the RE algebra. 
\end{example}

Let $\Z(\mathcal{\mathcal{A}})$ denote the Drinfeld centre of $\mathcal{A}$. The braiding $b: T \overset{\sim}{\Rightarrow} T\circ \mathrm{flip}$ of $\mathcal{A}$ provides a braided tensor functor $\mathcal{A} \rightarrow \Z(\mathcal{A})$ by sending $x\in \mathcal{A}$ to $(x, b_{-,x})$. Similarly, the reverse braiding provides a braided tensor functor $\mathcal{A}^\mathrm{rev} \rightarrow \Z(\mathcal{A})$. Together, they combine to form a braided tensor functor
\begin{align}\label{eq:center-functor}
    &\mathcal{A}\boxtimes \mathcal{A}^\mathrm{rev} \rightarrow \Z(\mathcal{A})\\ 
    &X\boxtimes Y \mapsto \left(X\otimes_\mathcal{A} Y, (\id{X} \otimes b_{Y,-}^{-1}) \circ (b_{-,X}\otimes \id{Y})\right)\nonumber
\end{align}
Under this braided tensor functor, the braided function algebra $\OA$ can be equipped with a half-braiding 
\begin{equation*}
    \tau_X: X\otimes\OA \xrightarrow{\sim} \OA\otimes X
\end{equation*}
according to \eqref{eq:center-functor} which we call the \textbf{field goal transform}. This turns $\OA$ into a commutative algebra in $\Z(\Ahat)$.

\begin{definition}\label{def:HC}
    The \textbf{Harish-Chandra category} is defined as the category of $\OA$-modules in $\Ahat$
    \[
        \HC{\mathcal{A}} := \OA\modu(\Ahat)~.
    \]
\end{definition}
Since $\OA$ is commutative in $\Z(\Ahat)$, any $\OA$-module in $\Z(\Ahat)$ (and by the forgetful functor in $\mathcal{A}$) inherits a $\OA\mbox{-}\OA$-bimodule structure. In particular, $\HC{\mathcal{A}}$ turns into a monoidal category via the relative tensor product over $\otimes_{\OA}$. 

\begin{definition}
    Let $A$ be an algebra in $\Ahat$ with multiplication $m: A \otimes A \rightarrow A$. A \textbf{quantum moment map} is an algebra morphism that is central, \ie \[ m \circ (\id{A} \otimes \mu) = m \circ (\mu\otimes \id{A})\circ \tau_A : A \otimes \OA \rightarrow A ~.\]
\end{definition}
A quantum moment map is the necessary datum for an algebra $A \in \Ahat$ to be an algebra in $\HC{\mathcal{A}}$ \cite[Prop.\ 3.7]{Safronov}.  

Given a algebra $A$ with a quantum moment map $\mu: \OA \rightarrow A$, the category of $A$-modules $A\modu(\Ahat)$ carries a right $\HC{\mathcal{A}}$-module structure via 
\begin{equation*}
    M \triangleleft X := M \otimes_{\OA} X
\end{equation*}
where $M$ is a seen as an $A\mbox{-}\OA$-bimodule using the quantum moment map $\mu$. 

\begin{definition}
    Let $A$ be an algebra with a quantum moment map $\mu$. An $A$-module $M$ is called \textbf{strongly equivariant} if $M$ is trivial as a right $\OA$-module. Strongly equivariant modules form a full subcategory
    \[A\modu(\Ahat)^\mathrm{str}\subset A\modu(\Ahat)~.\]
\end{definition}

Let $\mathcal{M}$ be a right $\mathcal{A}$-module category with action 
\[
-\triangleleft- : \mathcal{M} \boxtimes \mathcal{A} \rightarrow \mathcal{M}~.
\]
In particular, for $M \in \mathcal{M}$ we a functor 
\begin{equation*}
    M\triangleleft - : \mathcal{A} \rightarrow \mathcal{M}
\end{equation*}
which admits a right adjoint when passing to the associated cocompletions 
\begin{equation*}
    \underline{\Hom}(M,-): \ind{\mathcal{M}} \rightarrow \Ahat~,
\end{equation*}
called the internal Hom. The internal endomorphism $\underline{\End}(M):= \underline{\Hom}(M,M)$ carries a natural algebra structure and we have the following reconstruction theorem: 
\begin{theorem}[See, for example \cite{BZBJ18a}]\label{thm:GOs}
    Let $\mathcal{M}$ be a right $\mathcal{A}$-module category and fix an  $\mathcal{A}$-progenerator $M\in \mathcal{M}$. We have the following equivalence of right $\mathcal{A}$-module categories,
    \begin{equation*}
        \ind{\mathcal{M}} \simeq \underline{\End}(M)\modu(\Ahat)~.
    \end{equation*}
\end{theorem}

\subsection{Skein theory basics}\label{subsec:skein-mods}
Let $\mathcal{A}$ be a ribbon tensor category over a commutative ring $\field$. We refer to \cite{Cooke-thesis} for a more extensive review. 

\begin{definition}
\leavevmode 
\begin{itemize}
	\item An $\mathcal{A}$-\textbf{labelling} on a surface $\Sigma$ is a finite collection of framed points, each labelled by objects in $\mathcal{A}$.
	\item An $\mathcal{A}$-\textbf{ribbon graph} in a 3-manifold $M$ is a finite oriented ribbon graph possibly containing coupons, where ribbons end on coupons or possibly the boundary. Ribbons are labelled by objects in $\mathcal{A}$ and coupons are labelled by compatible $\mathcal{A}$-morphisms. In particular, an $\mathcal{A}$-ribbon graph in $M$ induces an $\mathcal{A}$-labelling on $\partial M$. 
	\item Let $M$ be a 3-manifold and $P$ a fixed $\mathcal{A}$-labelling on the boundary $\partial M$. The \textbf{relative $\mathcal{A}$-skein module} $\skmod_\mathcal{A}(M;P)$ is defined by
	\begin{equation*}
		\skmod_\mathcal{A}(M;P) := \field\langle \mathcal{A}\text{-ribbon graphs in }M\rangle /\sim
	\end{equation*}
	modulo isotopy and local relations by RT graphical calculus.
\end{itemize}
\end{definition}

\begin{definition}
\leavevmode 
	\begin{itemize}
		\item The \textbf{skein category} $\skcat_\mathcal{A}(\Sigma)$ of a surface $\Sigma$ consists of $\mathcal{A}$-labellings on $\Sigma$ as objects and morphism spaces are defined by skeins in the cylinder extending the $\mathcal{A}$-labellings, \ie 
		\begin{equation*}
			\Hom_{\skcat_\mathcal{A}(\Sigma)}(P, P') := \skmod_\mathcal{A}(\Sigma\times [0,1]; P\times \{0\}, P'\times \{1\})~.  
		\end{equation*}
		\item Let $M: \Sigma_\mathrm{in} \rightarrow \Sigma_{\mathrm{out}}$ be the 3-dimensional bordism. The construction of the skein module is functorial in the $\mathcal{A}$-labellings of the boundary by stacking and hence we associate the bimodule functor: 
		\begin{equation*}
			\skmod_\mathcal{A}: \skcat_\mathcal{A}(\Sigma_\mathrm{in})^\mathrm{op}  \boxtimes \skcat_\mathcal{A}(\Sigma_\mathrm{out}) \rightarrow \Mod_\field
		\end{equation*}
	\end{itemize}
\end{definition}
The above constructions assemble into a 3d oriented (once-categorified) TQFT:
\begin{theorem}[\cite{Walker-notes}]
	Skein theory forms a symmetric monoidal functor: 
	\[\operatorname{Sk}_\mathcal{A}: \operatorname{Bord}_{3,2}^\mathrm{or} \rightarrow \operatorname{Bimod}\]
\end{theorem}
The free cocompletion functor $\widehat{(-)}: \operatorname{Bimod} \rightarrow \prl$ from \eqref{def:cocompletion} further gives the $\prl$-valued TQFT: 
\begin{equation*}
	\skhat_{\mathcal{A}} := \widehat{(-)}\circ \operatorname{Sk}_\mathcal{A}: \operatorname{Bord}_{3,2}^{\mathrm{or}}\rightarrow \prl~.
\end{equation*}

\begin{remark}
	Skein theory is closely related to factorization homology. Namely, in \cite{Cooke-thesis,BrownHaioun} it is proven that 
	\begin{equation*}
		\skcat_\mathcal{A}(\Sigma) \simeq \int_\Sigma \mathcal{A}~, \quad
		\widehat{\skcat}_\mathcal{A}(\Sigma) \simeq \int_\Sigma \widehat{\mathcal{A}}~.
	\end{equation*}
	Conjecturally, a similar identity also holds on the level of 3-manifolds by using the $\beta$-version of factorization homology \cite{AFRbeta}. We note, however, that our work is independent of this conjecture.
\end{remark}

\subsection{Deformation quantization and holonomic modules}\label{subsec:def-quant-hol-mods}
In this section, we recall some important notions for deformation quantization modules including the condition of holonomicity, following \cite{KS}. Unless stated otherwise, we will work over $\Chbar$ under the base change $q\mapsto e^\hbar$.

A $\Chbar$-module $M$ is called $\hbar$-\textbf{torsion-free} if $\hbar:M\rightarrow M$ is injective, or equivalently if the $\hbar$-localisation map 
\[M\rightarrow \hloc{M}:=\Chloc\otimes_{\Chbar} M\]
is injective. 
The $\hbar$-completion of a $\Chbar$-module  $M$ is the module 
\begin{equation*}
    \hcmpl{M} := \lim_n M/\hbar^n M~.
\end{equation*}
The module $M$ is called $\hbar$-\textbf{complete} if the canonical map $M\rightarrow \hcmpl{M}$ is an isomorphism.
 
\begin{definition}
	Let $A_0$ be a commutative $\C$-algebra and $A_0\llbracket\hbar\rrbracket$ the associated $\Chbar$-algebra. A \textbf{star-product} $\star$ on $A_0\llbracket\hbar\rrbracket$ is an associative $\Chbar$-linear multiplication such that for all $a,b \in A_0$ 
	\begin{equation*}
		a \star b = \sum_{i\geq 0}{P_i(a,b)\, \hbar^i} \in A_0\llbracket\hbar\rrbracket~
	\end{equation*}  
	with the following properties: 
	\begin{enumerate}
		\item Each $P_i(\mbox{-},\mbox{-})$ is a bi-differential operator,
		\item $P_0(a,b) = a b \in A_0$,
		\item $P_i(a,1) = P_i(1,a) = 0$ for all $i>0$ and $a\in A_0$. 
	\end{enumerate}
\end{definition}
This endows $A_0$ with a Poisson bracket defined by: 
\begin{equation*}
	\{a,b\} := P_1(a,b) - P_1(b,a) = \hbar^{-1}(a\star b - b \star a) ~ \mathrm{mod} ~ \hbar
\end{equation*}
We say the star-algebra $(A_0\llbracket\hbar\rrbracket,\star)$ \textbf{deformation quantises} the Poisson algebra $(A_0, \{\cdot,\cdot\})$. 

\begin{definition}
A $\Chbar$-algebra $A$ is a \textbf{deformation quantization algebra} if it is isomorphic to a star-algebra $(A_0\llbracket\hbar\rrbracket,\star)$ for some commutative algebra $A_0$. In particular, such algebras are $\hbar$-torsion-free and $\hbar$-complete. 
\end{definition}
Let $X$ be an algebraic variety over $\C$. A \textbf{deformation quantization algebra} $A_X$ over $X$ is a deformation quantization algebra of $\mathcal{O}_X$. In particular, it induces a Poisson structure on $X$.

\begin{definition}
\leavevmode 
\begin{itemize}
\item A finitely generated $A_X$-submodule $M'$ of a finitely generated $\hloc{A_X}$-module $M$ is called an $A_X$-\textbf{lattice} if $\hloc{(M')}\cong M$. 
\item Let $M$ and $M'$ be as above. The \textbf{singular support} $SS(M;M')$ is defined as the support of $M'/\hbar$ in $X$.  
\end{itemize}
\end{definition}

\begin{proposition}{\cite[Prop.\ 2.3.18]{KS}}\label{prop:Gabber}
    The singular support $SS(M;M')$ is a coisotropic subvariety in $X$. 
\end{proposition}
Note that the dimension of $SS(M;M')$ is independent of the lattice. In particular, if $X$ is symplectic we can define holonomic modules as follows. 
\begin{definition}\label{def:holonomic}
    Let $X$ be a symplectic variety with a deformation quantisation algebra $A_X$. A finitely generated $\hloc{A_X}$-module $M$ is called \textbf{holonomic} if there exists an $A_X$-lattice $M'$ such that $SS(M;M')$ is Lagrangian.  
\end{definition}
By Proposition~\ref{prop:Gabber} holonomicity reduces to a criterion of half-dimensionality. Namely, we require 
\begin{equation}\label{eq:half-dim}\dim SS(M;M')= \frac{1}{2}\dim(X)~.\end{equation}
 
We now rephrase the above dimension criterion in terms of two different notions of dimension: the \textit{grade} defined by vanishing of Ext modules, and the \textit{Gelfand-Kirillov (GK) dimension} which measures the asymptotic growth based on filtrations. We follow the conventions made in \cite{Goodearl-Zhang}. 

Let $A$ be a Noetherian algebra over $\field$. 
\begin{definition}
    Let $M$ be a non-zero finitely generated $A$-module. Its \textbf{\textbf{grade}} is defined as the non-negative integer 
\[ \grade(M): = \inf\{i \mid \Ext^i_A (M,A)\neq 0\}~.\]
\end{definition}
Clearly, $j(M)$ is bounded from above by the projective dimension $\pdim(M)$. By definition $\Ext^\bullet_A(M,A)$ is bounded in $[\grade(M),\pdim(M)]$.  The algebra $A$ is said to satisfy the \textit{Auslander condition} if for any finitely generated $A$-module $M\neq 0$, any submodule $N\subset \Ext^i_A(M,A)$ satisfies 
\[\Ext^j_A(N,A) =0 \quad \forall j <i~.\]
In particular, $j(N)\geq  j(M)$.  If, in addition, $A$ has finite global dimension\footnote{There is no ambiguity between left and right global dimensions as $A$ is assumed to be Noetherian on both sides.} resp.\ finite injective dimension then it is called \textit{Auslander-regular} resp. \textit{Auslander-Gorenstein}.  

We now introduce the notion of GK dimensions, for which we assume that $A$ is also finitely generated. Let $V$ be a finite dimensional space of generators for $A$ and define the $\mathbb{Z}_{\ge 0}$-filtration $A_{\leq n}\subset A$ by the total degree of words. For an extensive exposition to GK dimensions we refer the reader to \cite{Krause-Lenagan} and \cite{McConnell-Robson}. 
\begin{definition}
The \textbf{GK dimension} of $A$ is defined as 
    \[
    \operatorname{GKdim}(A):= \inf\{d\mid \exists c:\dim_\field(A_{\leq n}) \leq c \cdot n^{d}\} \in \mathbb{R}_{\geq 0} \cup \{\infty\}
    \]
\end{definition}
If $\GKdim(A) = 0$, then $A$ is finite dimensional. The algebras of interest in this paper will be infinite dimensional of finite GK dimension, \ie of \textit{polynomial growth}. 

Similarly, if $M$ is a finitely generated $A$-module with a fixed choice of generators $m_1,\dots,m_r$ define the filtration
\[
M_{\leq n}:= A_{\leq n}\{m_1,\dots,m_r\} \subset M
\]
\begin{definition}
    The GK dimension of $M$ is defined as 
    \[
    \GKdim(M) := \inf\{d\mid \exists c : \dim(M_{\leq n}) \leq c \cdot n^{d}\}
    \]
\end{definition}
The algebra $A$ is called \textit{Cohen-Macaulay} if 
\begin{equation*}
        \grade(M) + \GKdim(M) = \GKdim(A)
\end{equation*}
for every non-zero finitely generated $A$-module $M$. In other words, Cohen-Macaulay algebras exhibit the grade of a module as the associated GK codimension.  In particular, if the algebra $A$ has integral GK dimension, so does any non-zero finitely generated module $M$. 

The following proposition is a collection of statements describing the behaviour of grades and GK dimensions for deformation quantization algebras. It relies on results from \cite{KS} and classical results of GK dimensions, which can be found in \cite{Krause-Lenagan}. 
\begin{proposition}\label{prop:DQ-GKdims}
    Let $A_X$ be a deformation quantization of a Poisson variety $X$, $M$ a finitely generated $\hloc{A_X}$-module and $M'$ the corresponding $\A_X$-lattice of $M$. Then, 
    \begin{enumerate}
        \item $A_X$ is Auslander-regular and Cohen-Macaulay, 
        \item $\GKdim(M) = \dim SS(M;M')$ and
        \item $\GKdim(A_X) = \dim X$.
    \end{enumerate}
\end{proposition}
\begin{proof}
    The fact that $A_X$ is a Auslander-regular is a direct consequence of Proposition 2.3.12 in \cite{KS}. This in combination with Proposition 2.3.11 in \cite{KS} also show the Cohen-Macaulay property. For the second statement we have $\GKdim(M) = \GKdim(M'/\hbar) = \dim \supp(M'/\hbar)=: \dim SS(M;M')$ where the first equality follows from \cite[Prop.\ 3.2]{Lezama-Venegas}. The last equality is obtained from $M=A_X$.  
\end{proof}
In particular, if $X$ is symplectic, by Proposition~\ref{prop:Gabber} and \eqref{eq:half-dim} a finitely generated $\hloc{A_X}$-module $M$ is holonomic if it satisfies $\GKdim(M) = \frac{1}{2}\dim(X)$ or equivalently $j(M) = \frac{1}{2}\dim(X)$.

\section{Internal skein algebras and modules with multiple gates}\label{sec:int-sk}

When we consider skein modules with multiple boundary components, and in particular when we consider the transfer bimodules for attaching separating 2-handles, we will require internal skein algebras for surfaces and 3-manifolds with multiple gates.

We end up with a generalisation of the handle-and-comb decomposition from \cite{BZBJ18a}, the details of which enter into the construction of 2-handle attaching maps in a crucial way.

\subsection{Internal skein algebras}
We now will introduce internal skein algebras which provide a kind of Barr-Beck reconstruction of skein categories, as categories of modules for algebras internal to $\mathcal{A}$. Topologically, internal skein algebras can be thought as generalisations of the ordinary skein algebra which allow skeins to end along a specified number of disks, so-called \textit{gates}, on the boundaries of the cylinder. 

The 2-disk $\mathbb{D}$ is an $\mathbb{E}_2$-algebra and the annulus $\Ann$ is an $\mathbb{E}_1$-algebra. The embedding $\mathbb{D}\hookrightarrow \Ann$ is an $\mathbb{E}_1$-algebra map. There is an obvious anti-involution by flipping $\Ann\rightarrow -\Ann$ which exhibits the $\mathbb{E}_1$-algebra isomorphism $\Ann \cong -\Ann^\mathrm{op}$.

Let $\Sigma$ be a surface with possibly non-empty boundary $\partial \Sigma$. Let $J = |\pi_0(\partial \Sigma)|$ denote the number of boundary components. In particular, $\Sigma$ inherits a canonical right $\Ann_{\partial \Sigma}$-module structure where $\Ann_{\partial \Sigma}$ is the $\mathbb{E}_1$-algebra obtained by
\begin{equation*}
    \Ann_{\partial \Sigma} := \Ann^{\amalg J}~.
\end{equation*}
A finite collection of disjoint disk embeddings $P = (P_1,\dots, P_n):  \bigcup_{i=1,\dots,n}\mathbb{D}_i\hookrightarrow \Ann_{\partial \Sigma}$ is an $\mathbb{E}_1$-algebra map and thus induces a right $\mathbb{D}_P $-module structure on $\Sigma$ where
\begin{equation*}
    \mathbb{D}_P := \bigcup_{i =1,\dots, n} \mathbb{D}_i~.
\end{equation*}
Factorisation homology thus induces a right $\mathcal{A}_P$-module structure on the skein category $\skcat_\mathcal{A}(\Sigma)$ where 
\begin{equation*}
    \mathcal{A}_P :=  \mathcal{A}^{\boxtimes n}~ = \int_{\mathbb{D}_P}{\mathcal{A}}~.
\end{equation*}
We refer to the disk embeddings $P_i$ above as \textbf{gates} for the surface $\Sigma$ and thus $n$ is the number of gates. Note that if $P$ is an empty collection, thus $n=0$, we have $\mathcal{A}_\emptyset = \Mod_\field$.

\begin{definition}
    Let $\Sigma$ be a surface and $P= (P_1,\dots, P_n)$ be a finite collection of gates. The $P$-\textbf{internal skein algebra} of $\Sigma$ is defined as the internal endomorphism algebra of $\skcat_\mathcal{A}(\Sigma)$ as a right $\mathcal{A}_P$-module category, \ie 
    \begin{equation*}
        \intskalg_{\mathcal{A},P}(\Sigma) := \underline{\End}_{\skcat(\Sigma)}(\emptyset) \in \Ahat_P~.
    \end{equation*}
\end{definition}
In terms of skeins, the $P$-internal skein algebra is obtained as a functor 
\begin{align*}
    \intskalg_{\mathcal{A},P}(\Sigma): &~\mathcal{A}_P^\mathrm{op}\rightarrow \Mod_\field\\ 
    &\boxtimes V_i \mapsto \skmod_\mathcal{A}(\Sigma\times [0,1], P(\boxtimes_i V_i), \emptyset) \nonumber
\end{align*}
thus allowing skeins to end on gates of the incoming boundary $\Sigma\times \{0\}$. 
When the ribbon category $\mathcal{A}$ is clear from context, we will denote the $P$-internal skein algebra by 
\begin{equation*}
    \A_{\Sigma, P} \equiv \intskalg_{\mathcal{A},P}(\Sigma)~. 
\end{equation*}
For $P = \emptyset$ we retrieve the ordinary skein algebra
\begin{equation*}
    \A_{\Sigma,\emptyset} = \skalg_\mathcal{A}(\Sigma)
\end{equation*}
and when $P:\mathbb{D} \rightarrow \Ann_{\partial \Sigma}$ is a single gate we obtain the internal skein algebra treated in \cite{GJS}. 

The following theorem exhibits compatibility between internal skein algebras with disjoint unions of surfaces, forgetting gates and merging gates (gluing 1-handles). Let $\varepsilon_i: \Ahat^{\boxtimes n} \rightarrow \Ahat^{\boxtimes n-1}$ denote the monoidal functor obtained by evaluating with $\unit \in \mathcal{A}$ on the $i$-th tensorant, \ie induced by the evaluation functor $\varepsilon: \Ahat \rightarrow \Mod_\field, X \mapsto X(\unit)$.  For distinct $i<j$ in $\{1,\dots ,n\}$, let $T_{ij}: \Ahat^{\boxtimes n} \rightarrow \Ahat^{\boxtimes n-1}$ be the monoidal functor induced by tensoring the $i$-th with the $j$-component. 
It is an extension by including multiple gates of Section~5.2 and Theorem 5.14 in \cite{BZBJ18a} and the proof follows verbatim. 

\begin{theorem}\label{thm:ISA-compatibility}
Let $\Sigma$ be a surface with a (possibly empty) collection of $n$ gates $P=(P_1,\dots, P_n)$. 
\begin{enumerate}
    \item Let $\Sigma'$ be another surface with a collection of $m$ gates $P' = (P'_1, \dots, P'_m)$. Then, \[\A_{\Sigma \cup \Sigma', P\cup P'}\cong \A_{\Sigma, P}\boxtimes \A_{\Sigma', P'}\] as algebras in $\Ahat_{P\cup P'} = \Ahat_P \boxtimes \Ahat_{P'}$.
    \item Let $n\geq 1$ and $\xi_iP = (P_1, \dots, \widehat{P}_i, \dots, P_n)$ be the collection of $n-1$ gates on $\Sigma$ by forgetting the $i$-th gate $P_i$. Then, \[\A_{\Sigma, \xi_iP} \cong \varepsilon_i(\A_{\Sigma, P})\] as algebras in $\Ahat_{\xi_i P} = \Ahat^{\boxtimes n-1}$. 
    \item For two distinct gates $P_i$ and $P_j$ consider the surface $\tilde{\Sigma}$ obtained from gluing a $1$-handle on $\Sigma$ along these two gates. Let $\tilde{P}$ be the collection of $n-1$ gates which consists of $P_k$'s for all $k\neq i,j$ and of the glued 1-handle seen as a single gate $P_{ij}$. Then, 
    \[
    \A_{\tilde{\Sigma},\tilde{P}} \cong T_{ij}(\A_{\Sigma,P})
    \]
    as algebras in $\Ahat^{\boxtimes n-1}$. 
\end{enumerate}
\end{theorem}
For example, the disk $\mathbb{D}$ with a single gate $P$ has a trivial internal skein algebra
\[
\A_{\mathbb{D}} \cong \widehat{\unit} = \Hom_\mathcal{A}(\mbox{-},\unit)\in \Ahat
\]
while the disk with two gates $P=(P_1,P_2)$ is 
\[
\A_{\mathbb{D}, P} \cong \OFRT = T^{R}(\unit) \in \Ahat^{\boxtimes 2}~.
\]
In particular, applying part 3 of Theorem~\ref{thm:ISA-compatibility} by merging the two gates $P_1, P_2$ on $\mathbb{D}$ along a 1-handle gives
\begin{equation*}
    \A_{\Ann} \cong T_{12}(\A_{\mathbb{D},P}) = T(T^R(\unit)) = \OA~.
\end{equation*}
\begin{remark}
More generally, Theorem~\ref{thm:ISA-compatibility} implies a formula for the internal skein algebra over any surface from the internal skein algebra of the disk $\mathbb{D}$ with two markings. Namely, any connected surface $\Sigma$ with non-empty boundary can be obtained by 1-handle attachments on $\mathbb{D}$. In particular, for a given system of gates $\A_{\Sigma,P}$ is obtained by multiple applications of part 3 on $\A_{\mathbb{D}, P'}$ for a suitable gate datum $ P'$ on $\mathbb{D}$ and a specified gluing pattern. 
\end{remark}

We introduce the following notational abbreviations: For $g,r \in \mathbb{Z}_{\geq 0}$
let $\Sigma_{g,r}^\ast$ denote the standard genus $g$ surface with $r$ boundary components $\Sigma_{g,r}$ equipped with a single gate on each boundary component. 
We denote the associated internal skein algebra 
\[\A_{g,r}:= \A_{\Sigma_{g,r}^\ast} \in \Ahat^{\boxtimes r}\] 

\begin{corollary}\label{cor:ISA-Sgb}
    Let $\Sigma_{g,r}$ be the genus $g$-surface with $r$ boundary components and a single gate $P$. Then
\[
    \A_{\Sigma_{g,r}; P} \cong \A_{g,1}^{\totimes g}\totimes\OA^{\totimes r -1}\in \Ahat
    \]
    as algebras. 
\end{corollary}
\begin{proof}
    This follows directly from Theorem~\ref{thm:ISA-compatibility} by picking a handle and comb decomposition of $\Sigma_{g,r}$, see Figure~\ref{fig:handle-and-comb-single-gate}.
\end{proof}

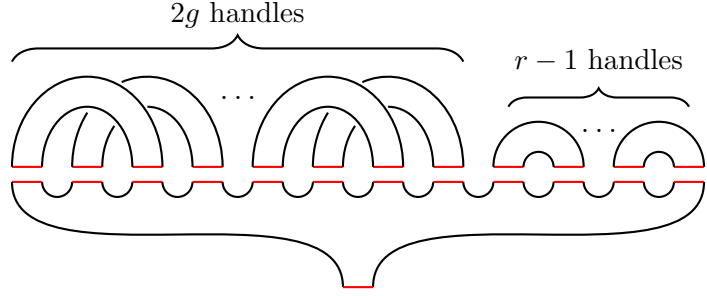
\begin{figure}
    \centering
    \begin{tikzpicture}[scale=0.4, thick]
\foreach \i in {0,2,...,22} {
  \draw[red] (\i,4) -- (\i+1,4);  
  \draw[red] (\i,3.5) -- (\i+1,3.5);  
  \ifthenelse{\i <22}{\draw (\i+1,3.5) arc (-180:0:0.5 and 0.5)}{};
}
\draw (0,3.5) to[out=270,in=90,looseness=0.8] (11,0);
\draw (23,3.5) to[out=270,in=90,looseness=0.8] (12,0);
\draw[red] (11,0) -- (12,0);
\foreach \i in {0,8} {
\draw (\i,4) arc (180:0:2.5 and 3);
\draw (\i+1,4) arc (180:0:1.5 and 2);
\draw (\i+2,4) arc (180:143:2.5 and 3);
\draw (\i+7,4) arc (0:110:2.5 and 3);
\draw (\i+3,4) arc (180:137:1.5 and 2);
\draw (\i+6,4) arc (0:90:1.5 and 2);
}
\foreach \i in {16,20} {
\draw (\i,4) arc (180:0:1.5 and 1.5);
\draw (\i+1,4) arc (180:0:0.5 and 0.5);
}
\node at (7.6,6.3) {$\dots$};
\draw [decorate,decoration={brace,amplitude=10pt}]
  (0,7.5) -- (15,7.5) node[midway,yshift=1.5em]{\small $2g$ handles};
  \node at (19.6,5.2) {$\dots$};
\draw [decorate,decoration={brace,amplitude=7pt}]
  (16.5,6) -- (22.5,6) node[midway,yshift=1.5em]{\small $r-1$ handles};
\end{tikzpicture}
    \caption{The handle and comb presentation of the surface $\Sigma_{g,r}$ with a single gate.}
    \label{fig:handle-and-comb-single-gate}
\end{figure}

\begin{figure}
    \centering
    \begin{tikzpicture}[scale=0.4, thick]
\foreach \i in {0,2,...,22} {
  \draw[red] (\i,4) -- (\i+1,4);  
  \draw[red] (\i,3.5) -- (\i+1,3.5);  
  \ifthenelse{\i < 7}{\draw (\i+1,3.5) arc (-180:0:0.5 and 0.5)}{};
  \ifthenelse{\i > 7 \AND \i <22\AND \NOT \i = 12}{\draw[Myblue] (\i+1,3.5) arc (-180:0:0.5 and 0.5)}{};
}
\draw (0,3.5) to[out=270,in=90,looseness=0.8] (6,0);
\draw (13,3.5) to[out=270,in=90,looseness=0.8] (7,0);
\draw[red] (6,0) -- (7,0);
\draw[Myblue] (14,3.5) to[out=270,in=90,looseness=0.8] (18,0);
\draw[Myblue] (23,3.5) to[out=270,in=90,looseness=0.8] (19,0);
\draw[red] (18,0) -- (19,0);
\foreach \i in {0,8,16} {
\ifthenelse{\i <16}{\draw}{\draw[Myblue]} (\i,4) arc (180:0:2.5 and 3);
\ifthenelse{\i <8}{\draw}{\draw[Myblue]} (\i+1,4) arc (180:0:1.5 and 2);
\ifthenelse{\i <8}{\draw}{\draw[Myblue]} (\i+2,4) arc (180:143:2.5 and 3);
\ifthenelse{\i <8}{\draw}{\draw[Myblue]} (\i+7,4) arc (0:110:2.5 and 3);
\ifthenelse{\i <8}{\draw}{\draw[Myblue]} (\i+3,4) arc (180:137:1.5 and 2);
\ifthenelse{\i <8}{\draw}{\draw[Myblue]} (\i+6,4) arc (0:90:1.5 and 2);
}
\end{tikzpicture}
    \caption{A handle and comb presentation of the surface $\Sigma_{g=2,r=2}$ with a single gate on each boundary component. The two boundary circles are differentiated by colour.}
    \label{fig:handle-and-comb-multi-gates}
\end{figure}
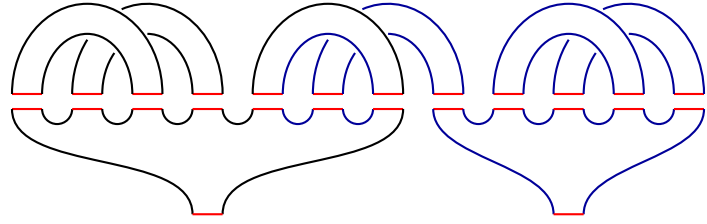

The algebra $\A_{1,1}$ of the once-punctured torus with one gate is obtained by the smash product of the annular $\OA$
\begin{equation}\label{eq:ISA-torus}
\D_\mathcal{A}:= \A_{1,1} \cong \OA\sharp \OA
\end{equation}

\begin{corollary}\label{cor:ISA-separating}
Let $\Sigma_{g,r}$ be the genus $g$-surface with $r$ boundary components with $P$ one gate on each boundary component. This induces an algebra isomorphism
\[
\A_{\Sigma_{g,r}^\ast} = \A_{g,r} \cong \A_{1,1}^{\totimes g}\totimes\A_{0,2}^{\totimes r-1} \in \Ahat^{\boxtimes r}.
\]
\end{corollary}
\begin{proof}
    This follows from applying Theorem~\ref{thm:ISA-compatibility} to a handle-and-comb decomposition which includes $r$ many combs (see Figure~\ref{fig:handle-and-comb-multi-gates}).
\end{proof}
The algebra associated to the annulus with one gate on each of the two boundary disks is the smash product 
\begin{equation}\label{eq:ISA-ann}
    \D'_{\mathcal{A}}:= \A_{0,2} \cong \OA \sharp \OFRT~.
\end{equation}
In particular, applying the tensor product we obtain $T(\D'_\mathcal{A}) = \D_\mathcal{A}$ as algebras in $\Ahat$. 

\begin{example}\label{ex:quantization}
For our main example, $\mathcal{A} =\Rep_q^\mathrm{fd} G$, $\D_\mathcal{A}= \OqG \sharp \OqG$ and $\D'_\mathcal{A} \cong \OqG \sharp \OqFRT$ have appeared many times in the literature as algebras of $q$-difference operators on $G$.  We use the notation $\DqG$ and $\Dq'(G)$, respectively, in this case to recall this appearance. An explicit presentation will be given for $G= \SL_2,\GL_2$ in Section~\ref{subsec:GL2-background}. 

The internal skein algebra of $\Sigma_{g,r}^\ast$ is a deformation quantization algebra of the Fock-Rosly Poisson structure on the moduli space of the space of flat connections on $\Sigma$ equipped with framings at each gate (see \cite{Fock-Rosly,AGS} for more details).   
\end{example}

Finally, modules over internal skein algebras reconstruct the value of skein theory on the associated surface: 
\begin{proposition}\label{prop:ISA-GOs}
Let $\Sigma$ be a possibly disconnected surface whose connected components have non-empty boundary. Consider a collection $P$ of gates on $\Sigma$ with at least one gate on each of the connected components. 
    Then we have an equivalence of right $\mathcal{A}_P$-module categories, 
    \[
    \skhat_{\mathcal{A}}(\Sigma) \simeq \A_{\Sigma,P}\modu(\Ahat_P)~.
    \]
\end{proposition}
\begin{proof}
Let $\emptyset \in \sk_{\mathcal{A}}(\Sigma)$ be the distinguished object, \ie the empty $\mathcal{A}$-labelling on $\Sigma$. It is an $\mathcal{A}_P$-progenerator with respect to the right $\mathcal{A}_P$-module structure as we have at least one gate at each connected component. The equivalence then follows by a direct application of Theorem~\ref{thm:GOs}. 
\end{proof}

The following is obtained using \eqref{eq:ISA-ann} and Definition~\ref{def:HC}. 
\begin{corollary}
    Skeins on the annulus recover the Harish-Chandra category 
    \[\skhat_\mathcal{A}(\Ann) \simeq \HC{\mathcal{A}}~.\]
\end{corollary}

Returning to a general surface $\Sigma$ with boundary $\partial \Sigma$, let $b_0 \in \pi_0(\partial \Sigma)$ be a fixed boundary component. Then, $\sk_{\mathcal{A}}(\Sigma)$ is naturally a right $\sk_{\mathcal{A}}(\Ann)$-module category by the right $\Ann$-action on $ \Sigma$ through the boundary component $b_0$.  In particular, $\skhat_{\mathcal{A}}(\Sigma)$ is a right $\HC{\mathcal{A}}$-module category and thus the algebra $\A_{\Sigma, P}$ associated to a single gate $P$ on the boundary $b_0$ inherits a quantum moment map
\[
\mu: \OA \rightarrow \A_{\Sigma,P}~.
\]
Restricting $\A_{\Sigma,P}$-modules to strongly equivariant modules with respect to the above quantum moment map recovers skein theory on the surface obtained by capping the $b_0$ boundary of $\Sigma$ with a disk: 
\begin{proposition}
    Let $\Sigma(b_0) = \Sigma \cup_{b_0} \mathbb{D}$ be the surface obtained by capping the $b_0$- boundary circle of $\Sigma$ with a disk. There is an equivalence
    \begin{equation*}
        \skhat_\mathcal{A}(\Sigma(b_0)) \simeq \A_{\Sigma,P}\modu(\Ahat)^\mathrm{str}~.
    \end{equation*}
\end{proposition}
The above proposition appears as Proposition 2.29 in \cite{GJS}. Although the proof there assumes that $\Sigma(b_0)$ is closed, this assumption is not used. 
In particular, if $\Sigma_g$ is the closed genus $g$ surface then  
\[
\skhat_{\mathcal{A}}(\Sigma_g) \simeq \A_{g,1}\modu(\Ahat)^\mathrm{str} ~.
\]
Similarly, $\A_{g,r}$ comes with a quantum moment map 
\begin{equation*}
    \mu: \OA^{\boxtimes r}=\O_{\mathcal{A}^{\boxtimes r}} \rightarrow \A_{g,r}~
\end{equation*}
and 
\[
\skhat_\mathcal{A}(\Sigma_g) \simeq \A_{g,r}\modu(\Ahat^{\boxtimes r})^{\mathrm{str}}~.
\]

\subsection{Internal skein modules}

Thus far, we have reconstructed skein theory on surfaces in terms of internal skein algebras and their modules. We now shift focus to skein theory on 3-manifolds, which naturally leads to the notion of \textit{internal skein modules}.

Let $M: \Sigma_\mathrm{in}\rightarrow \Sigma_\mathrm{out}$ be an oriented bordism between closed surfaces $\Sigma_\mathrm{in}, \Sigma_\mathrm{out}$. Skein theory assigns to $M$ the bimodule
\[
\sk_\mathcal{A}(M):\sk_\mathcal{A}(\Sigma_\mathrm{in})\opp\boxtimes \sk_\mathcal{A}(\Sigma_\mathrm{out})\rightarrow \Mod_\field~.
\]
Let $\delta: (\mathbb{D}^\circ)^{\amalg b}\hookrightarrow \Sigma_\mathrm{in}$ and $\delta': (\mathbb{D}^\circ)^{\amalg b'}\hookrightarrow \Sigma_\mathrm{out}$ be non-empty collections of disjoint open disk embeddings in $\Sigma_\mathrm{in}$ resp.\ $\Sigma_\mathrm{out}$, with at least one disk embedding on each connected component of $\Sigma_\mathrm{in}$ and $\Sigma_\mathrm{out}$. Let $\Sigma_\mathrm{in}^{\delta}$ and $\Sigma_\mathrm{out}^{\delta'}$ denote the associated complements, \ie the surface with the specified disks removed, which are surfaces with $b$ resp.\ $b'$ boundary components.  Restricting $\sk_\mathcal{A}(M)$ along $\sk_\mathcal{A}(\Sigma_\mathrm{in}^\delta) \rightarrow\sk_\mathcal{A}(\Sigma_\mathrm{in})$ and $\sk_\mathcal{A}(\Sigma_\mathrm{out}^{\delta'}) \rightarrow\sk_\mathcal{A}(\Sigma_\mathrm{out})$ gives an object in $\skhat_\mathcal{\mathcal{A}}(\Sigma_\mathrm{in}^\delta)\boxtimes\skhat_\mathcal{A}(-\Sigma_\mathrm{out}^{\delta'})$. For any gates $P_\mathrm{in}\subset\Ann_{\partial \Sigma_\mathrm{in}^\delta}$ and $P_\mathrm{out}\subset\Ann_{\partial \Sigma_\mathrm{out}^{\delta'}}$ with at least one gate for each connected component, Proposition~\ref{prop:ISA-GOs} gives
\begin{equation}\label{eq:ISA-Bimod}
 \skhat_\mathcal{\mathcal{A}}(\Sigma_\mathrm{in}^\delta)\boxtimes\skhat_\mathcal{A}(-\Sigma_\mathrm{out}^{\delta'}) \simeq \A_{\Sigma_\mathrm{in}^\delta;P_\mathrm{in}}\mbox{-}\A_{\Sigma_\mathrm{out}^{\delta'};P_\mathrm{out}}\mbox{-}\Bimod(\Ahat_{P_\mathrm{in}}\boxtimes \Ahat_{P_\mathrm{out}})~.
\end{equation}
We have used Theorem~\ref{thm:ISA-compatibility} for
\[\A_{\Sigma_\mathrm{in}^{\delta}\cup-\Sigma_{\mathrm{out}}^{\delta'};P_{\mathrm{in}}\cup P_{\mathrm{out}}}\cong \A_{\Sigma_{\mathrm{in}}^{\delta};P_{\mathrm{in}}} \boxtimes \A_{-\Sigma_{\mathrm{out}}^{\delta'};P_{\mathrm{out}}}\cong \A_{\Sigma_{\mathrm{in}}^{\delta};P_{\mathrm{in}}} \boxtimes\A_{\Sigma_{\mathrm{out}}^{\delta'};P_{\mathrm{out}}}^{\mathrm{op}}~. \]
\begin{definition}
    Let $P$ denote the datum $(\delta, \delta';P_\mathrm{in},P_\mathrm{out})$ of disk removals and gates as described above. The \textbf{$P$-internal skein bimodule} of $M$ is defined as the image of $\sk_\mathcal{A}(M)$ under the equivalence \eqref{eq:ISA-Bimod}
    \[
    \intskmod_{\mathcal{A},P}(M) \in \A_{\Sigma_\mathrm{in}^\delta;P_\mathrm{in}}\mbox{-}\A_{\Sigma_\mathrm{out}^{\delta'};P_\mathrm{out}}\mbox{-}\Bimod(\Ahat_{P_\mathrm{in}}\boxtimes \Ahat_{P_\mathrm{out}})~.
    \]
\end{definition}
Equivalently, it is a left $\A_{\Sigma_\mathrm{in}^\delta;P_\mathrm{in}}\boxtimes\A_{\Sigma_\mathrm{out}^{\delta'};P_\mathrm{out}}\opp$-module in $\Ahat_P$, where $\A_{\Sigma_\mathrm{out}^{\delta'};P_\mathrm{out}}\opp$ denotes the braided opposite algebra. 
\begin{remark}
    One can also allow empty gates on the boundary of $M$. This will lead to the ordinary skein (bi)module $\skmod_\mathcal{A}(M)$ over the ordinary skein algebras $\skalg_\mathcal{\mathcal{A}}(\Sigma_\mathrm{in})$ and $\skalg_\mathcal{A}(\Sigma_\mathrm{out})$. 
\end{remark}

By construction, internal skein modules constructed from bordisms between closed surfaces will be strongly equivariant with respect to the quantum moment maps arising on the boundaries of the surfaces with discs removed. 

Cutting a bordism along a surface is compatible with internal skein bimodules as long as we keep at least one boundary in each connected component of that surface. This is formulated in the following proposition and corollary which are a mild generalisation of \cite[Thm.\ 4.1]{GJS}:

\begin{proposition}
    Let $M: \Sigma_1\rightarrow \Sigma_2$ and $N: \Sigma_2\rightarrow \Sigma_3$ be composable bordisms and let $(\delta_i; P_i)$ be the data of disk removals and gates on $\Sigma_i$. If $P_2$ has at least one gate on each connected component of $\Sigma_2^{\delta_2}$, then there is an isomorphism of $\left(\A_{\Sigma_1^{\delta_1},P_1},A_{\Sigma_3^{\delta_3}, P_3}\right)$-bimodules: 
    \begin{equation*}
        \intskmod_{\mathcal{A},(P_1,P_3)}(N\circ M) \cong \Hom_{\Ahat_{P_2}}\left(\unit,\intskmod_{\mathcal{A},(P_1, P_2)}(M)\otimes_{\A_{\Sigma_2^{\delta_2}, P_2}} \intskmod_{\mathcal{A},(P_2, P_3)}(N)\right)~.
    \end{equation*}
\end{proposition}

\begin{corollary}\label{cor:int-bimod-gluing}
    If $\Ahat$ has trivial M\"uger centre, then 
    \[    
    \intskmod_{\mathcal{A},(P_1,P_3)}(N\circ M) \cong \intskmod_{\mathcal{A},(P_1, P_2)}(M)\otimes_{\A_{\Sigma_2^{\delta_2}, P_2}} \intskmod_{\mathcal{A},(P_2, P_3)}(N)~.\]
\end{corollary}

\section{Skein transfer bimodules}\label{sec:transfer bimodules}
In this section, we construct transfer bimodules for skein theory which are used to construct analogues of inverse and direct images. Throughout this section, fix $\mathcal{A} = \Rep^{\mathrm{fd}}_q (G)$ and $q$ generic. 

Consider a bordism $M:\Sigma_1\rightarrow \Sigma_2$ between closed surfaces and let $\delta_1$ and $\delta_2$ be disc removal data on $\Sigma_1$ and $\Sigma_2$. Recall from example \ref{ex:quantization} that the internal skein algebras $\A(\Sigma_i^{\delta_i})$ quantize the Poisson varieties $R_G(\Sigma_i).$ The internal skein bimodule $\intskmod_{q,G}(M)$ is a quantization of the Lagrangian correspondence 
\[ R_G(M) \rightarrow R_G(\Sigma_1)\times R_G(\Sigma_2)~.\]

We will use this feature of internal skein modules to define \textit{transfer functors} (see Definition~\ref{def:transfer-functor}). Before giving the definition we recall briefly compression bodies as by Heegaard decomposition they are the building blocks for 3-bordisms. 
\begin{figure}
	\centering
    \begin{tikzpicture}[scale = 1.2]
        \draw[thick] (-3,0) to[out=90, in=90, looseness = 0.7] (3,0);
        \draw[thick] (-3,0) to[out=-90, in=-90, looseness = 0.7] (3,0);
        \def\g{2} 
        \foreach \i in {1,...,\g} {
            \pgfmathsetmacro\x{(\i-0.5)*6/\g - 3}
            \draw[thick] (\x+0.5,0) arc[start angle=30, end angle=150, radius=0.5];
            \draw[thick] (\x+0.57,0.2) arc[start angle=0, end angle=-180, radius=0.5];
        }
		\node at (0.3,0) {$\cdots$};
            \draw[thick, red] (-1.5,-0.3) to[out=180, in=180] (-1.5,-1.07);
            \draw[thick, red,dotted] (-1.5,-0.3) to[out=0, in=0] (-1.5,-1.07);
            \node at (-1.9,-0.6) {$\alpha$};
            \draw[thick, Myblue] (-0.5,1.2) to[out=180, in=180,looseness=0.3] (-0.5,-1.2);
            \draw[thick, Myblue,dotted] (-0.5,1.2) to[out=0, in=0,looseness=0.3] (-0.5,-1.2);
            \node at (-0.2,-0.6) {$\gamma_{1,g-1}$};
    \end{tikzpicture}
\caption{Attaching non-separating curve $\alpha$ and separating curve $\gamma_{1,g-1}$ on the genus $g$-surface $\Sigma_g$.}
\label{fig:handle-gluing}
\end{figure}
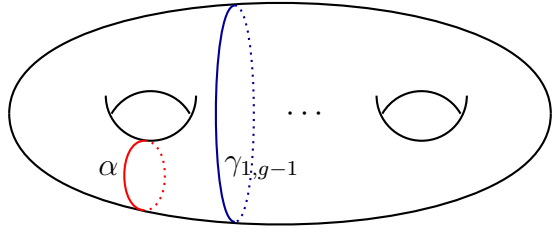

\begin{definition}
    A \textbf{compression body} $C$ is obtained from a cylinder $\Sigma\times I$ by gluing $2$-handles along a specified multicurve $\alpha\subset \Sigma \times \{0\}$ and potentially filling $3$-handles along boundary $2$-spheres.
\end{definition}

In other words, a compression body is specified by a triple $(\Sigma, \alpha,  X)$ consisting of a surface $\Sigma$, a multicurve $\alpha \subset \Sigma$, and a subset $X$ of boundary $2$-spheres in $\Sigma_\alpha$, the surface resulting from $\Sigma$ after gluing 2-handles along $\alpha$. Such a compression body defines a bordism 
\begin{equation*}
    C: \partial_{\mathrm{in}} C \rightarrow \Sigma
\end{equation*}

For example, a handlebody is a compression body such that $\partial_\mathrm{in}C = \emptyset$.

The \textit{change of coordinates principle} \cite[Ch.\ 1.3]{FarbMargalit} states that:
\begin{quote}
    Two simple closed curves $\alpha, \beta \subset \Sigma$ are related by an (orientation) preserving homeomorphism $f: \Sigma \rightarrow \Sigma$, \ie $f(\alpha) = \beta$, if and only if the cut surfaces $\Sigma\backslash\alpha$ and $\Sigma\backslash\beta $ are homeomorphic. 
\end{quote}

As a consequence, in our proofs we will restrict our attention to the following 2-handle attachments as depicted in Figure~\ref{fig:handle-gluing}.  We call these the \textbf{standard curves}:
     
\begin{enumerate}
    \item (Non-separating $C_\alpha$) Attachment along the non-separating curve $\alpha$.
    \item (Separating $C_{\gamma_{g_1,g_2}}$) Attachment along a separating curve $\gamma_{g_1,g_2}$ which partitions the surface into surfaces of genus $g_1,g_2$ with $g=g_1+g_2$. 
\end{enumerate}

\begin{proposition}\label{prop:compression}
Every compression body $C = (\Sigma, \alpha, X)$ as above may be obtained as a composition of compression bodies $C_i$ which glue a single 2-handle along a standard curve, together with mapping cylinders of diffeomorphisms and 3-handles closing boundary spheres.  Namely, we have:

\begin{equation*}
    C \cong M_{f_1}\circ C_{1}\circ \dots \circ M_{f_m} \circ C_m\circ X, 
\end{equation*}
where each $C_i$ is a standard curve $C_\alpha$ or $C_{\gamma_{g_1,g_2}}$, and $X$ is a collection of 3-handles.
\end{proposition}  

\begin{proof} By definition a compression body $C = (\Sigma, \alpha, X)$ may be obtained iteratively by a composition of compression bodies which glue a single 2-handle.

Namely, if $\alpha = \alpha_1\cup \dots \cup\alpha_m$ consists of $m$ disjoint loops $\alpha_i$ on $\Sigma$, then 
\begin{equation*}
    C \cong C'_{1}\circ \dots \circ C'_m\circ X 
\end{equation*}
where $C'_{i}: \Sigma_{i}\rightarrow \Sigma_{i-1}$ are compression bodies obtained gluing a single 2-handle, defined as follows: Set $\Sigma_{0}:= \Sigma$ and $C'_1: \Sigma_1 \rightarrow \Sigma$ is the compression body obtained from gluing a 2-handle along $\alpha_1\subset \Sigma\times\{0\}$. Iteratively, let $C'_i: \Sigma_{i}\rightarrow \Sigma_{i-1}$ be the compression body gluing a single 2-handle along $\alpha_i$ on $\Sigma_{i-1} = \partial_{\mathrm{in}}C'_{i-1}$.

By the change of coordinates principle, we may replace each $\alpha_i$ by some $f_i$ applied to a standard curve.  The result is to replace $C'_i$ by $M_{f_i}\circ C_i$ where now $C_i$ is a compression body attaching along a standard curve $\alpha$ or $\gamma_{g_1,g_2}$.
\end{proof}

Given a general compression body $C = (\Sigma, \alpha,X)$ of a closed surface $\Sigma$, fix a gate datum $P$ consisting of a single disk removal and a single gate for each connected component of $\partial_{\mathrm{in}}C$ excluding sphere components.  This choice further defines a disk removal and gate datum $P'$ on $\partial_\mathrm{out}C =\Sigma$ which might include multiple gates on a single connected component.

\begin{remark}
We do not treat the filling of any resulting boundary 2-spheres in the compression body, since these do not alter the skein module.  That is however a feature of the (underived) skein module, and one expects the derived skein module to be sensitive to the presence of boundary 2-spheres, or to second homology in $M$ more generally.
\end{remark}

Returning to skeins, given a $3$-bordism $M$ we define the associated transfer functor via its internal skein bimodule: 
\begin{definition}\label{def:transfer-functor}
    The \textbf{transfer functor} $\mathcal{T}_M$ associated to the bordism $M: \Sigma_1 \rightarrow \Sigma_2$ is defined as 
    \[
    \mathcal{T}_M:= \intskmod_{q,G}(M)\otimes_{\A_{\Sigma_2}} \mbox{-}: \A_{\Sigma_2}\modu(\Ahat_{P_2}) \rightarrow \A_{\Sigma_1}\modu(\Ahat_{P_1})~.
    \]
\end{definition}
By Corollary~\ref{cor:int-bimod-gluing}, the transfer functor $\mathcal{T}_M$ is itself functorial in $M$. Thus, by Heegaard decomposition the building blocks such transfer functors arise from compression bodies: 
\begin{definition}
Let $C = (\Sigma,\alpha,X)$ be a compression body and $(P,P')$ fixed gate data as described above. The transfer bimodule of $C$ is the $(P,P')$-internal skein bimodule
\[
\A(C):= \intskmod_{\mathcal{A};(P,P')}(C)\in \A_{\partial_\mathrm{in}C;P}\mbox{-}\A_{\Sigma; P'}\mbox{-}\Bimod(\Ahat^{\boxtimes 2\mid\pi_0(\partial_\mathrm{in}C)\mid })~.
\]
\end{definition}
By Corollary~\ref{cor:int-bimod-gluing} and Proposition~\ref{prop:compression} we have 
\[
\A(C) \cong \A(C_m)\otimes_{\A_{\Sigma_{m-1}}}\cdots \otimes_{\A_{\Sigma_1}} \A(C_1)
\]
as $\A_{\partial_{in}C}\mbox{-}\A_{\Sigma}$-bimodules. Hence, transfer bimodules of compression bodies are determined by single $2$-handle attachments. 

We will write $C_\alpha$ and $C_{\gamma_{g_1,g_2}}$ for the compression bodies obtained by the associated 2-handle attachments. For their transfer functors, we will use the suggestive notation 
\begin{equation*}
    i^\ast:=\mathcal{T}_{C_\alpha}: \A_{g,r}\modu(\Ahat^{\boxtimes r}) \rightarrow \A_{g-1,r}\modu(\Ahat^{\boxtimes r})~
\end{equation*}
and
\begin{equation*}
    \delta_{g_1,g_2}:= \mathcal{T}_{C_{\gamma_{g_1,g_2}}}: \A_{g,r}\modu(\Ahat^{\boxtimes r}) \rightarrow \A_{g_1;r_1}\boxtimes\A_{g_2;r_2}\modu(\Ahat^{\boxtimes r})~.
\end{equation*}
This is justified as the former exhibits the inverse image along the closed immersion $i:\{e\}\times G^{g-1}\hookrightarrow G^g$. The associated transfer bimodules are given by
\begin{align}
    \A_{g-1\rightarrow g}:=\A(C_\alpha)&\cong (A-I)\backslash\DqG \totimes\Dq(G)^{g-1} \totimes \DqFRT^{\totimes r-1}~,\label{eq:ns-transfer}\\ 
    \A_{g\rightarrow (g_1,g_2),r}:=\A(C_{\gamma_{g_1,g_2}})&\cong \DqG^{\totimes g}\totimes (A-I)\backslash\DqFRT\totimes \DqFRT^{\totimes r-2}~.\label{eq:s-transfer}
\end{align}
where we write $(A-I)$ as shorthand for the right ideal kernel of the counit map $\epsilon:\OqG\rightarrow \field$ for the associated $G$-cycle. The right $\A_{g,r}$-action on the bimodules is induced from the right regular action on $\A_{g,r}$ while the left $\A_{g-1,r}$ resp.\ $\A_{g_1,r_1}\boxtimes \A_{g_2,r_2} $ comes from acting on the right with the remaining cycles while braiding past the $(A-I)\backslash\DqG\cong \OqG$ factor resp.\ $(A-I)\backslash\DqFRT\cong \OqFRT$ factor.

\section{Holonomicity of skein modules}\label{sec:MainThm}

In this section we build up to a geometric approach to internal skein algebras and transfer bimodules.   We recreate key elements of Sabbah's approach to $q$-difference operators, and using them to define the notion of holonomicity for skein modules, and to establish preservation under handle attachments.

\subsection{Algebras of \texorpdfstring{$q$}{q}-difference operators on \texorpdfstring{$\GL_2$}{GL2} and \texorpdfstring{$\SL_2$}{SL2}}\label{subsec:GL2-background}

We now recall in detail the algebras $\DqG$ and $\DqFRT$ of $q$-difference operators on $G$ for $G=\GL_2,\SL_2$ and their basic properties. These have appeared in many places, e.g. \cite{KlimykS}, \cite{Backelin-Kremnizer} and \cite{VV}. We follow the conventions of \cite{Jquiver}, \cite{BJ} and \cite{BalagovicJ}.

\begin{definition}
	The Hopf algebra $\Uqgl{2}$ is generated by $E,F, K_i^\pm$ for $i=1,2$ with defining relations: 
	\begin{align*}
		[K_1, E]_q &= 0 & [K_1,F]_{q^{-1}} &= 0 \\ 
		[K_2, E]_{q^{-1}} &= 0 & [K_2,F]_{q} &= 0 \\ 
		[K_1,K_2]&=0 & [E,F] &= \frac{K_1 K_2^{-1} - K^{-1}K_2}{q- q^{-1}}  
	\end{align*}
	The Hopf algebra structure is given by: 
	\begin{align*}
		\Delta E &= E\otimes K_1 K_2^{-1} + 1\otimes E & S E &= - E K_1^{-1} K_2 \\
		\Delta F &= F\otimes 1 + K_1^{-1} K_2\otimes F & S F &= - K_1 K_2^{-1} F \\
		\Delta K_i &= K_i \otimes K_i & S K_i &= K_i^{-1}   
	\end{align*}
\end{definition}
The $R$-matrix on the fundamental representation $V(1)$ is given by 
\begin{equation}\label{eq:R-matrix}
    R = \begin{pmatrix}
        q & 0& 0& 0\\ 
        0 & 1 & 0 & 0 \\ 
        0 & q - q^{-1} &1 &0 \\ 
        0 & 0 & 0 & q
    \end{pmatrix}
\end{equation}

\begin{remark}
    The subalgebra in $U_q\mathfrak{gl}_2$ generated by $E,F$ and $K = K_1 K_2^{-1}$ retrieves the quantum group $U_q\mathfrak{sl}_2$. The associated $R$-matrix on $V(1)$ differs from \eqref{eq:R-matrix} by a $q^{-1/2}$ factor: 
    \begin{equation*}
        \tilde{R}:= q^{-1/2} R 
    \end{equation*}
\end{remark}

An equivalent description of $\Uqgl{2}$ is generated by the entries of: 
\begin{equation*}
	L^+ = \begin{pmatrix}
		l_1^{+1} & l_2^{+1}\\
		0 & l_2^{+2}
	\end{pmatrix},
	\quad 
	L^- = \begin{pmatrix}
		l_1^{-1} & 0\\
		l_1^{-2} & l_2^{-2}
	\end{pmatrix}
\end{equation*}
satisfying the following relations: 
\begin{align*}
	&l_1^{+1} l_1^{-1} = l_2^{+2}l_2^{-2} = 1 & &[l_1^{\pm1}, l_2^{\pm2}] = 0\\ 
	&[l_1^{+1} ,l_2^{+1}]_{q^{-1}}=0  & &[l_2^{+2}, l_2^{+1}]_{q} =0\\ 
	&[l_1^{+1} ,l_1^{-2}]_{q} = 0 & &[l_2^{+2}, l_2^{-2}]_{q^{-1}} = 0\\ 
	&[l_1^{-2}, l_2^{+1}] = (q -q^{-1}) (l_2^{+2}l_1^{-1} - l_1^{+1} l_2^{-2})   
\end{align*}

We now recall the construction of the quantum coordinate algebras $\OqGL{2}$ and $\OqSL{2}$ as \textit{reflection equation algebras}.

\begin{definition}\label{def:Oqmat}
	Let $\OqMat$ be the algebra generated by the entries of $L = \begin{pmatrix}
		l_1^1 & l_2^1\\ 
		l_1^2 & l_2^2 
	\end{pmatrix}$
satisfying the reflection equation
\[R_{21} L_1 R L_2 = L_2 R_{21} L_1 R\]
or explicitly in terms of generators
\begin{align*}
	[l_2^{1}, l_1^{1}] &= -\numberq{-2}{q}\,l_2^1 l_2^2  & &[l_2^{2}, l_1^{1}] = 0\\
	[l_1^{2}, l_1^{1}] &=  \numberq{-2}{q}\,l_2^2 l_1^2  & &[l_2^{2},l_2^{1}]_{q^2} =0\\ 
	[l_1^{2}, l_2^{1}] &= -\numberq{-2}{q}\,(l_1^1 l_2^2 -l_2^2 l_2^2)  & &[l_2^{2},l_1^{2}]_{q^{-2}} =0~. 
\end{align*}
\end{definition}
Recall our convention $\numberq{n}{q}:= (q^{n}-1)$ for $n\in \mathbb{Z}$ and $q\in \field^{\times}$.
The $q$-determinant $\det_q(L) := l_1^1 l_2^2 -q^2 l_2^1 l_1^2$ is a central element in $\OqMat$ and hence we can make the following definition: 

\begin{definition}
	The quantum coordinate algebra of $\GL_2$ is the localisation: 
	\begin{equation*}
		\OqGL{2} := \OqMat\left[\detq(L)^{-1}\right]~
	\end{equation*}
    while the quantum coordinate algebra of $\SL_2$ is the specialisation: 
    \begin{equation*}
        \OqSL{2} := \OqMat/(\detq(L) -1)
    \end{equation*}
\end{definition} 

\begin{proposition}[\cite{JL92}]
    There is a unique algebra embedding $\phi: \OqGL{2} \hookrightarrow \Uqgl{2}$ given by $L\mapsto L^+ S(L^-)$ whose image is in the locally finite part of $\Uqgl{2}$. The locally finite part is given by adjoining the element $l_1^{+1} l_2^{+2}$ to the image of $\phi$. 

    For $\SL_2$, the embedding $\phi: \OqSL{2}\hookrightarrow \Uqsl{2}$ is an isomorphism onto the locally finite part of $\Uqsl{2}$.
\end{proposition}

\begin{lemma}
	The centre of $\OqGL{2}$ coincides with the subspace of $\Uqgl{2}$-invariants with respect to the adjoint action. It is the subalgebra generated by the $q$-trace 
	\begin{equation*}
		\tr_q(L) := l_1^1+ q^{-2} l_2^2
	\end{equation*}
	 and $\det_q(L)^{\pm 1}$.
\end{lemma}

\begin{definition}
Let $\DqMat{\varepsilon}$ for $\varepsilon = 0,1$ be the algebra generated by the entries of
\begin{equation*}
	A=\begin{pmatrix}
		a_1^1 & a_2^1 \\ 
		a_1^2 & a_2^2,
	\end{pmatrix}
	\quad
    \mathrm{and}
    \quad
	D=\begin{pmatrix}
		\partial_1^1 & \partial_2^1 \\ 
		\partial_1^2 & \partial_2^2,
	\end{pmatrix}
\end{equation*}
satisfying the relations
\begin{align*}
	R_{21} A_1 R A_2 &= A_2 R_{21} A_1 R \\
	R_{21} D_1 R D_2 &= D_2 R_{21} D_1 R\\
	R_{21} D_1 R A_2 &= A_2 R_{21} D_1 R_{21}^{-1}~.
\end{align*}
The first two equations make $a^i_j$ and $\partial^i_j$ subject to the $\OqMat$-relations as in Definition~\ref{def:Oqmat}, while the last equation is expressed explicitly as 
\begin{align}
	&[\partial^1_1,a^1_1]_{q^{-2+\varepsilon}} = \numberq{-2}{q}\,  \partial^1_2 a^2_1 + q^{\varepsilon-4}\numberq{2}{q}\, a^1_2  \partial^2_1 \nonumber\\
	&[\partial^1_1,a^1_2]_{q^{-2+\varepsilon}} = \numberq{-2}{q}\, \partial^1_2 a^2_2 \nonumber\\
	&[\partial^1_1, a^2_1]_{q^{\varepsilon}} =  -\numberq{2}{q}\, \partial^2_1 a^1_1 - q^{-2}\numberq{2}{q}^2\,  \partial^2_2 a^2_1 - q^{\varepsilon}\numberq{-2}{q}\, a^2_2 \partial^2_1 \nonumber\\
	&[\partial^1_1 ,a^2_2]_{q^{\varepsilon}} =  -\numberq{2}{q}\,  \partial^2_1 a^1_2 -q^{-2}\numberq{2}{q}^2\, \partial^2_2 a^2_2 \nonumber\\
	&[\partial^1_2,a^1_1]_{q^{\varepsilon}} = -q^{\varepsilon}\numberq{-2}{q}\, a^1_2 \partial^2_2 + q^{\varepsilon}\numberq{-2}{q}\, a^1_2 \partial^1_1 \nonumber\\
	&[\partial^1_2,a^1_2]_{q^{-2+\varepsilon}} = 0 \nonumber\\
	&[\partial^1_2 ,a^2_1]_{q^{\varepsilon}} = \numberq{-2}{q}\,  \partial^2_2 a^1_1 + q^{\varepsilon}\numberq{-2}{q}\, a^2_2 \partial^1_1 - q^{\varepsilon}\numberq{-2}{q}\, a^2_2 \partial^2_2 \nonumber\\
	&[\partial^1_2,a^2_2]_{q^{-2+\varepsilon}} = \numberq{-2}{q}\, \partial^2_2 a^1_2 \nonumber\\
	&[\partial^2_1,a^1_1]_{q^{-2+\varepsilon}} =  \numberq{-2}{q}\,\partial^2_2 a^2_1\nonumber\\
	&[\partial^2_1, a^1_2]_{q^{\varepsilon}} = \numberq{-2}{q}\,\partial^2_2 a^2_2  \nonumber\\
	&[\partial^2_1,a^2_1]_{q^{-2+\varepsilon}} = 0 \nonumber\\
	&[\partial^2_1, a^2_2]_{q^{\varepsilon}} = 0 \nonumber\\
	&[\partial^2_2, a^1_1]_{q^{\varepsilon}} = -q^{\varepsilon}\numberq{2}{q}\,a^1_2\partial^2_1\nonumber\\
	&[\partial^2_2,a^1_2]_{q^{\varepsilon}} = 0\nonumber\\ 
	&[\partial^2_2,a^2_1]_{q^{-2+\varepsilon}} = q^{\varepsilon}\numberq{-2}{q}\,a^2_2 \partial^2_1\nonumber \\ 
	&[\partial^2_2,a^2_2]_{q^{-2+\varepsilon}} = 0 
	\nonumber~.
	\end{align}
\end{definition}
In other words, $\DqMat{1}$ is obtained from $\DqMat{0}$ by modifying its relations so that every summand of the form $a^i_j  \partial^k_l$ is multiplied by $q$. 

The braided tensor product $\DqMat{\varepsilon}\totimes\DqMat{\varepsilon}$ is the algebra with generators given by the entries of the 2x2-matrices: 
\begin{equation*}
	A = (a^i_j), \quad D = (\partial^i_j), \quad \tilde{A} = (\tilde{a}^i_j), \quad \tilde{D} = (\tilde{\partial}^i_j)
\end{equation*} 
corresponding to the two copies of $\DqMat{\varepsilon}$ with the additional relations from the braided tensor product: 
\begin{equation}\label{eq:braided-Oq-Oq}
	\tilde{B}_1 R B_2 R^{-1} = R B_2 R^{-1} \tilde{B}_1
\end{equation}
for $B \in \{A,D\}$ and $\tilde{B}\in \{\tilde{A},\tilde{D}\}$. Explicitly, 
\begin{align}
	&[\tilde{b}^1_1, b^1_1] =\numberq{-2}{q}\,  \tilde{b}^1_2 b^2_1 \nonumber\\
	&[\tilde{b}^1_1,b^1_2]_{q^{-2}} =-\numberq{-2}{q}\, b^1_1 \tilde{b}^2_1 + \numberq{-2}{q}\, \tilde{b}^1_2 b^2_2 \nonumber\\
	&[\tilde{b}^1_1,b^2_1]_{q^{2}} = 0 \nonumber\\
	&[\tilde{b}^1_1,b^2_2] =\numberq{-2}{q}\,  b^2_1 \tilde{b}^1_2 \nonumber\\
	&[\tilde{b}^1_2, b^1_1] = 0 \nonumber\\
	&[\tilde{b}^1_2, b^1_2] = 0 \nonumber\\
	&[\tilde{b}^1_2, b^2_1] = 0 \nonumber\\
	&[\tilde{b}^1_2, b^2_2] = 0\nonumber\\
	&[\tilde{b}^2_1, b^1_1] = \numberq{-2}{q}\,(\tilde{b}^2_2 - \tilde{b}^1_1) b^2_1\nonumber\\
	&[\tilde{b}^2_1,b^1_2] =  \numberq{-2}{q}\,\left(\left(\tilde{b}^2_2-\tilde{b}^1_1\right)b^2_2 + b^1_1 \left(\tilde{b}^1_1 - \tilde{b}^2_2\right)\right)\nonumber\\
	&[\tilde{b}^2_1, b^2_1] = 0 \nonumber\\
	&[\tilde{b}^2_1, b^2_2] = -\numberq{-2}{q}\, b^2_1 \left(\tilde{b}^2_2 - \tilde{b}^1_1\right) \nonumber\\
	&[\tilde{b}^2_2, b^1_1] = \numberq{-2}{q}\,\tilde{b}^1_2 b^2_1\nonumber\\
	&[\tilde{b}^2_2,b^1_2]_{q^{2}} = -\numberq{-2}{q}\, \left(b^1_1 \tilde{b}^1_2 - \tilde{b}^1_2 b^2_2\right)\nonumber\\ 
	&[\tilde{b}^2_2,b^2_1]_{q^{-2}} = 0 \nonumber \\ 
	&[\tilde{b}^2_2,b^2_2] =  \numberq{-2}{q}\,b^2_1 \tilde{b}^1_2 
	\nonumber 
	\end{align}
These relations determine the $g$-th braided tensor product $\left(\DqMat{\varepsilon}\right)^{\totimes g}$. 

It is an easy computation to show that $q$-determinants $\detq(A)$ and $\detq(D)$ satisfy for any $i,j\in\{1,2\}$
\begin{equation}\label{eq:detq-commutation}
    \partial^i_j\, \detq(A) = q^{2(\varepsilon-1)} \detq(A) \, \partial^i_j\quad \mathrm{and} \quad \detq(D)\, a^i_j = q^{2(\varepsilon-1)} a^i_j\, \detq(D)~.
\end{equation}
In particular, the following Proposition is immediate.
\begin{proposition}
The $q$-determinants $\det_q(A)$ and $\det_q(D)$ are central in $\DqMat{1}$ and they generate an Ore set in $\DqMat{0}$, \ie
	\begin{equation*}
		S:= \{\operatorname{det}_q(A)^n\operatorname{det}_q(D)^m\mid n,m \in \mathbb{Z}_{\geq 0}\} \subset \DqMat{0}~.
	\end{equation*} 
\end{proposition}

Localising the $q$-determinants in $\DqMat{0}$ gives $\DqGL{2}= \OqGL{2}\sharp\OqGL{2}$ while specialising $\DqMat{1}$ to $\detq(A) = 1$ and $\detq(D)= 1$ gives $\DqSL{2} = \OqSL{2}\sharp\OqSL{2}$. 

\begin{definition}
	The algebra of $q$-difference operators on $\GL_{2}$ is the localisation 
	\begin{equation*}
		\DqGL{2} := S^{-1}\DqMat{0}~.
	\end{equation*}
    The algebra of $q$-difference operators on $\SL_2$ is the specialisation 
    \begin{equation*}
        \DqSL{2} := \DqMat{1}/(\detq(A)-1, \detq(D)-1)
    \end{equation*}
\end{definition}

\begin{remark}
    The algebras $\DqSL{2}$ and $\DqGL{2}$ is what is called $\A_{1,1}$ in Section \ref{sec:int-sk}; see Corollary \ref{cor:ISA-separating} and Equation \ref{eq:ISA-torus}.
\end{remark}

We also give the concrete presentation for the algebras $\OqFRT$ and $\DqFRT = \OqG \sharp \OqFRT$ for $G=\GL_2, \SL_2$ whose relations are less involved. 
\begin{definition}
    The FRT algebra on $2\times 2$ matrices $\OqMatb$ is generated by the entries of 
    \[F = \begin{pmatrix}
        f^1_1 & f^1_2\\ 
        f^2_1 & f^2_2
    \end{pmatrix}\]
    subject to the relation
    \[R_{21} F_1F_2 = F_2 F_1 R\]
    or explicitly 
    \begin{align*}
        &[f^1_1, f^1_2]_q = 0 &&[f^1_2, f^2_1] = (q^{-1}-q)f^2_2 f^1_1\\
        &[f^1_1, f^2_1]_{q^{-1}} = 0 &&[f^1_2, f^2_2]_{q^{-1}} = 0\\
        &[f^1_1, f^2_2] = 0  &&[f^2_1, f^2_2]_q = 0~.
    \end{align*}
\end{definition}
The $q$-determinant in $\OqMatb$ is the central element \[\detq(F):= f^1_1f^2_2 - q f^1_2 f^2_1 ~.\]
\begin{definition}
    The FRT algebra $\OqGLb{2}$ is the localisation 
    \[\OqGLb{2}:= \OqMatb[\detq(F)^{-1}]\]
    while the algebra $\OqSLb{2}$ is the specialization
    \[\OqSLb{2}:= \OqMatb/(\detq(F)-1)~.\]
\end{definition}

\begin{definition}
    The algebra $\DqMatb{\varepsilon}$ for $\varepsilon=0,1$ is the algebra generated by the entries
    \[
    A= \begin{pmatrix}
        a^1_1 & a^1_2\\ 
        a^2_1 &a^2_2
    \end{pmatrix}
    \quad 
    \mathrm{and}
    \quad
    F= \begin{pmatrix}
        f^1_1 & f^1_2\\ 
        f^2_1 &f^2_2
    \end{pmatrix}\]
    subject to relations
\begin{align*}
	R_{21} A_1 R A_2 &= A_2 R_{21} A_1 R \\
	R_{21} F_1  F_2 &= F_2 F_1 R\\
	F_2 A_1  &= R_{21} A_1 R F_2~.
\end{align*} 
Explicitly, the last equation translates for $ j \in\{1,2\}$ to
\begin{align*}
    &[f^1_j,a^1_1]_{q^{2-\varepsilon}} = q^{-\varepsilon}\numberq{2}{q}\, a^1_2 f^2_j\\ 
    &[f^1_j,a^1_2]_{q^{1-\varepsilon}} = 0\\
     &[f^1_j,a^2_1]_{q^{1-\varepsilon}} = q^{-\varepsilon-1}\numberq{2}{q}\,a^2_2 f^2_j\\ 
    &[f^1_j,a^2_2]_{q^{-\varepsilon}} = 0\\
     &[f^2_j,a^1_1]_{q^{-\varepsilon}} = q^{-\varepsilon}\numberq{2}{q}\, a^2_1 f^1_j + q^{-\varepsilon-2}\numberq{2}{q}^2 \,a^2_2 f^2_j\\ 
    &[f^2_j,a^1_2]_{q^{1-\varepsilon}} = q^{-\varepsilon-1}\numberq{2}{q}\,a^2_2 f^1_j\\
     &[f^2_j,a^2_1]_{q^{1-\varepsilon}} = 0\\ 
    &[f^2_j,a^2_2]_{q^{2-\varepsilon}} = 0~.
\end{align*}
\end{definition}
 The $q$-determinants $\detq(A) = a^1_1a^2_2 - q^2 a^1_2 a^2_1$ and $\detq(F) = f^1_1f^2_2 - qf^1_2f^2_1$ are central in $\DqMatb{1}$ while in $\DqMatb{0}$ they $q^2$-commute in the following way
 \begin{equation}\label{eq:detqF-commutation}
     f^i_j \,\detq(A) = q^2 \detq(A)\, f^i_j\quad \text{and }\quad \detq(F)\,a^i_j = q^2 a^i_j\, \detq(F)
 \end{equation}
 In particular, we can define $\DqGLb{2}$ and $\DqSLb{2}$ in the following way. 
 \begin{definition}
     The algebra $\DqGLb{2}$ is the localisation 
     \[\DqGLb{2}:= \DqMatb{0}[\detq(A)^{-1}\detq(F)^{-1}]\]
     while the algebra $\DqSLb{2}$ is the specialisation 
     \[\DqSLb{2} := \DqMatb{1}/(\detq(A)-1, \detq(F)-1)~.\]
 \end{definition}

\begin{remark}
    The algebras $\DqSLb{2}$ and $\DqGLb{2}$ are what are called $\A_{0,2}$ in Section \ref{sec:int-sk}; see Corollary \ref{cor:ISA-Sgb} and Equation \eqref{eq:ISA-ann}.
\end{remark}
 
Equation~\eqref{eq:braided-Oq-Oq} gave a presentation for the braided tensor product $\DqMat{\varepsilon}\totimes\DqMat{\varepsilon}$. Similarly, the braided tensor product $\DqMat{\varepsilon}\totimes \DqMatb{\varepsilon}$ is generated by the entries of
\[A=(a^i_j),\quad D=(\partial^i_j), \quad \tilde{A}=(\tilde{a}^i_j),\quad \tilde{F}=(\tilde{f}^i_j)\]
with braided relations for $B\in \{A,D\}$
\[\tilde{A}_1 R B_2 R^{-1} = RB_2R^{-1}\tilde{A}_1 \]
and
\begin{equation}\label{eq:braided-Oq-OqFRT}
    \tilde{F}_1 B_2 = R B_2 R^{-1} \tilde{F}_1~.
\end{equation}
Equation $\eqref{eq:braided-Oq-OqFRT}$ is given explicitly for $j\in\{1,2\}$ as 
\begin{align*}
    &[\tilde{f}^1_j,a^1_1] = 0\\ 
    &[\tilde{f}^1_j,a^1_2]_{q} = 0\\ 
    &[\tilde{f}^1_j,a^2_1]_{q^{-1}} = 0\\ 
    &[\tilde{f}^1_j,a^2_2] = 0\\ 
    &[\tilde{f}^2_j,a^1_1] = q^{-2}\numberq{2}{q}a^2_1 \tilde{f}^1_j\\ 
    &[\tilde{f}^2_j,a^1_2]_{q^{-1}} = q\numberq{-2}{q}(a^1_1 -a^2_2)\tilde{f}^1_j\\ 
    &[\tilde{f}^2_j,a^2_1]_q = 0\\ 
    &[\tilde{f}^2_j,a^2_2] = -\numberq{2}{q}a^2_1 \tilde{f}^1_j
\end{align*}
 
Consider the following algebra maps of $\OqG$ in $\DqG$ resp.\ $\DqFRT$ for $G=\GL_2,\SL_2$: 
\begin{align}
	&\alpha: \OqG\hookrightarrow \DqG,~ L\mapsto A\label{eq:alpha-embd}\\
	& \gamma:\OqG \hookrightarrow \DqFRT, ~L \mapsto A\label{eq:gamma-embd}\\
	& \mu: \OqGL{2}\rightarrow \DqGL{2},~ L\mapsto DA^{-1}D^{-1}A \label{eq:moment-embd}
\end{align}
Here $L$, $A$ and $D$ denote the matrix of generators that we have been using. 
In terms of internal skein algebras these maps correspond to the annuli embeddings as follows: The map $\alpha$ in \eqref{eq:alpha-embd} corresponds to the embedding of the annulus in the once-punctured torus along the first handle-attachment. The map $\gamma$ in \eqref{eq:gamma-embd} corresponds to the the annulus embedding in the with two gates. The map $\mu$ in \eqref{eq:moment-embd} corresponds to the annulus embedding in the once-punctured torus along the boundary and is the \textbf{quantum moment map}.

\begin{lemma}
	We have the following square of embeddings and all algebra morphisms are flat: 
	\[\begin{tikzcd}
		\OqMat \arrow[r,hookrightarrow, "\alpha"]\arrow[d,hookrightarrow]& \DqMat{0} \arrow[d, hookrightarrow]\\
		\OqGL{2} \arrow[r, "\alpha", hookrightarrow]& \DqGL{2}
	\end{tikzcd}
	\]
\end{lemma}
\begin{proof}
    The canonical localisation morphisms $\OqMat \rightarrow \OqGL{2}$ resp.\ $\DqMat{0} \rightarrow \DqGL{2}$ are flat, since localisation is flat, and they are injective as $\OqMat$ resp.\ $\DqMat{0}$ are domains. Flatness of $\alpha$ follows because each algebra $\DqMat{0}$, $\DqGL{2}$ being defined as a twisted tensor product is in fact free over each factor. 
\end{proof}

\begin{notation}
We will mostly treat the two cases $G=\GL_2,\SL_2$ in parallel for the remainder of the paper, and so we will write 
\[\A_{g,r} = \DqG^{\totimes g} \totimes \DqFRT^{\totimes r-1}
\] 
to mean both algebras, whenever a statement holds for both. 
Furthermore, we will write 
\[\Aplus_{g,r}=\DqMat{\varepsilon}^{\totimes g}\totimes \DqMatb{\varepsilon}^{\totimes r-1}\]
where $\varepsilon = 0$ for $ \GL_2$ and $\varepsilon =1 $ for $ \SL_2$. The two cases will be invoked separately only within proofs. In particular, in this notation: 
\begin{enumerate}
\item $\A_{g,r} = \Aplus_{g,r}[\detq^{-1}]$ for $\GL_2$ (inverting all $q$-determinants) and 
\item $\A_{g,r} = \Aplus_{g,r}/\detq-1$ for $\SL_2$ (specialising all $q$-determinants). 
\end{enumerate}
For $\GL_2$ localisation defines the functor
\[\detq^{-1}: \Aplus_{g,r}\modu \rightarrow \A_{g,r}\modu\]
and for $\SL_2$ we have the equivalence between the subcategory of $\Aplus_{g,r}$-modules (scheme-theoretically) supported at $\detq =1 $ and the category of $\A_{g,r}$-modules
\[
\Aplus_{g,r}\modu^{\detq =1} \simeq \A_{g,r}\modu~.
\]
\end{notation}

\begin{proposition}\label{prop:flat-ann-embd}
	The algebra map obtained from any annular embeddings (see \eqref{eq:alpha-embd} and \eqref{eq:gamma-embd}) 
	\[
	\OqMat\hookrightarrow \Aplus_{g,r}\]
	is flat for the algebras.
\end{proposition}
\begin{proof}
It suffices to check the statement at $q=1$ which then follows from the classical $\O(\Mat_2)\hookrightarrow \O\left(\Mat_2^{2(g+r-1)}\right)$ which corresponds to a (flat) projection map $\Mat_2^{2(g+r-1)}\rightarrow \Mat_2$.
\end{proof}

\subsection{Holonomicity of \texorpdfstring{$\A_{g,r}$-modules}{modules}}

The internal skein algebras $\A_{g,r}$ are deformation quantisation algebras of the Fock-Rosly Poisson variety $X:= G^{2(g+r-1)}$ \cite{STS,Fock-Rosly, Alekseev00}. While the latter is not everywhere symplectic, it contains an open dense symplectic leaf $X_s: = \mu^{-1}(G^{\ast})$ \cite{GanevJS} on which all strongly equivariant modules -- and in particular all internal skein modules -- are supported. 

For the rest of this paper, unless otherwise specified, we will always work over $q$ generic. Thus, by abuse of previous notation, we write $\A_{g,r}$ instead of $\A_{g,r}^{\mathrm{loc}}$ as $q$ is declared to be generic. In particular, we call a finitely generated $\A_{g,r}$-module $M$ holonomic if $\GKdim(M) = (g+r-1)\dim(G)$. This is equivalent to asking that $SS(M) \cap X_s$ is Lagrangian in $X_s$, see Definition~\ref{def:holonomic} and Proposition~\ref{prop:DQ-GKdims}.

\begin{proposition}\label{prop:Agr-holonomic} We have the following characterisations of holonomic $\A_{g,r}$-modules by comparing with $\Aplus_{g,r}$:
\begin{enumerate}
    \item  For $\GL_2$, a finitely generated $\A_{g,r}$-module $M$ is holonomic if and only if there exists a finitely generated $\Aplus_{g,r}$-submodule $M'$ such that 
    \begin{equation*}
        M \cong \A_{g,r} \otimes_{\Aplus_{g,r}} M'
    \end{equation*}
    and $\GKdim(M') \leq 4(g+r-1)$.
    \item For $\SL_2$, a holonomic $\A_{g,r}$-module $M$ is the same as a finitely generated $\Aplus_{g,r}$-module scheme-theoretically supported at $\detq = 1$ and with dimension $\GKdim(M) = 3(g+r-1)$.
\end{enumerate}
\end{proposition}
\begin{proof}
    \begin{enumerate}
\item If $M'$ is an $\Aplus_{g,r}$-lattice of $M$ satisfying $\GKdim(M')\leq 4(g+r-1)$, then $j(M)\geq j(M')$ by flatness of localisation and thus $\GKdim(M)\leq \GKdim(M')\leq 4(g+r-1)$ by Proposition~\ref{prop:DQ-GKdims}. On the other hand, consider a holonomic $\A_{g,r}$-module $M$ and pick a finitely generated $\Aplus_{g,r}$-lattice $M'$. In fact, by the ascending chain condition there is a unique largest $M'$ such that $\GKdim(M') = 4(g+r-1)$. 
\item This follows tautologically from the definition. 
    \end{enumerate}
\end{proof}

\begin{remark}\label{rem:thick-subcat}
The full subcategory $\A_{g,r}\modu_{\mathrm{hol}}$ of holonomic $\A_{g,r}$-modules forms a thick\footnote{A full subcategory is called \textit{thick} if is closed under subobjects, quotients and extensions.} subcategory of $\A_{g,r}\modu_{\mathrm{fg}}$, the category of finitely generated $\A_{g,r}$-modules. Similarly, the derived category $D^b_{\mathrm{hol}}(\A_{g,r})$ of bounded complexes in $\A_{g,r}\modu_{\mathrm{fg}}$ with holonomic cohomologies is thick in $D^b(\A_{g,r})$.
\end{remark}

\subsection{The Koszul resolution and the inverse image}\label{subsec:GL2-Koszul}

The algebra $R:= \OqMat$ is a (quadratic) PBW algebra in the sense of \cite{quadratic} with the ordered generating set:
\[\{x_1 := a^2_2 < x_2 := a_1^1 < x_3 := a_2^1 < x_4 := a_1^2\}\]
The Koszul dual $R^!$ is then a PBW algebra with ordered generators: 
\[\{y_1 := x_4^\ast  < y_2 := x_3^\ast < y_3 := x_2^\ast < y_4 := x_1^\ast\}\]
and relations: 
\begin{align*}
	&y_1^2 = y_2^2 = y_3^2 = 0 &
	  y_4^2 &= -\numberq{-2}{q} y_1 y_2 \\ 
	 &\{y_2,y_1\} = 0 &
	  \{y_3, y_1\} &= 0\\
	  &\{y_4,y_1\}_{q^2} =  -\numberq{-2}{q}y_1 y_3& 
	  \{y_3, y_2\} &=0  
	 \\ 
	 &\{y_4, y_2\}_{q^{-2}} =  \numberq{-2}{q}q^{-2} y_2 y_3&
	 \{y_4,y_3\} &= \numberq{-2}{q}y_1 y_2
\end{align*}

Let $A$ be the matrix of generators of $R = \OqMat$. Following standard methods for Koszul algebras \cite{quadratic}, we can find the following Koszul resolution of the right $R$-module $\field = (A-I)\backslash R$: 
\begin{equation}\label{eq:Koszul-GL2}
    0 \rightarrow R \rightarrow R^{\oplus 4} \rightarrow R^{\oplus 6} \rightarrow R^{\oplus 4} \rightarrow R
\end{equation}
with differentials
\begin{align}\label{eq:Koszul-differentials-GL2}
\delta(\widehat{x}_1)&= a^2_2 -1\nonumber\\
\delta(\widehat{x}_2)&= a^1_1 -1\nonumber\\
\delta(\widehat{x}_3)&= a^1_2\nonumber\\
\delta(\widehat{x}_4)&= a^2_1\nonumber\\
\delta(\widehat{r}_1)&= \widehat{x}_1(a^1_1-1) - \widehat{x}_2(a^2_2-1)\nonumber\\
\delta(\widehat{r}_2)&= \widehat{x}_1a^1_2 - \widehat{x}_3 q^{2}(a^2_2-q^{-2})\nonumber\\
\delta(\widehat{r}_3)&= \widehat{x}_1a^2_1 - \widehat{x}_4 q^{-2}(a^2_2-q^{2})\nonumber\\
\delta(\widehat{r}_4) &= \widehat{x}_2 a^1_2 +\widehat{x}_3 \left(q^{-2}\numberq{2}{q}a^2_2 - (a^1_1-1)\right) \nonumber\\ 
\delta(\widehat{r}_5) &= \widehat{x}_2 a^2_1 -\widehat{x}_4 \left(q^{-2}\numberq{-2}{q}a^2_2 + (a^1_1-1)\right)\nonumber\\ 
\delta(\widehat{r}_6)&= \widehat{x}_1\numberq{-2}{q}a^2_2  - \widehat{x}_2\numberq{-2}{q}a^2_2 + \widehat{x}_3 a^2_1 - \widehat{x}_4 a^1_2\nonumber\\ 
\delta(\widehat{c}_1)&=\widehat{r}_1 a^1_2 + \widehat{r}_2\left(q^{-2} \numberq{2}{q}a^2_2 - (a^1_1 -1 )\right)+ \widehat{r}_4 q^{2}(a^2_2-q^{-2})\nonumber\\
\delta(\widehat{c}_2)&= \widehat{r}_1 a^2_1- \widehat{r}_3 \left(q^{-2}\numberq{-2}{q}a^2_2 + (a^1_1 -1)\right)+ \widehat{r}_5 q^{-2}(a^2_2-q^{2})\nonumber\\
\delta(\widehat{c}_3)&= \widehat{r}_1q^{-2}\numberq{2}{q}a^2_2 + \widehat{r}_2 a^2_1 -\widehat{r}_3 a^1_2+ \widehat{r}_6(a^2_2 -1)\nonumber\\ 
\delta(\widehat{c}_4)&= \widehat{r}_1 q^2\numberq{-2}{q}a^2_2 + \widehat{r}_4 a^2_1 -\widehat{r}_5 a^1_2 + \widehat{r}_6 (a^1_1 -1)\nonumber \\ 
\delta(\widehat{k}) &= \widehat{c}_1 a^2_1  - \widehat{c}_2 a^1_2 + \widehat{c}_3 (a^1_1 -1)- \widehat{c}_4 (a^2_2-1) 
\end{align}
where $\{\widehat{k}\}$, $\{\widehat{c}_1,\dots,\widehat{c}_4\}$, $\{\widehat{r}_1,\dots,\widehat{r}_6\}$ and $\{\widehat{x}_1,\dots ,\widehat{x}_4\}$ denote $R$-basis of $R$, $R^{\oplus4}$, $R^{\oplus 6}$ and $R^{\oplus 4}$ (reading from left) in the Koszul resolution \ref{eq:Koszul-GL2}. 
At $q=1$, one retrieves the classical Koszul resolution.
 
Using flatness of the algebras (see Proposition \ref{prop:flat-ann-embd})
\begin{equation*}
	\OqMat \hookrightarrow \Aplus_{g,r} \hookrightarrow \A_{g,r}
\end{equation*}
in the $\GL_2$-case,
we further obtain a Koszul resolution
\begin{equation*}
	0 \rightarrow \A_{g,r} \rightarrow \A_{g,r}^{\oplus 4} \rightarrow \A_{g,r}^{\oplus 6} \rightarrow \A_{g,r}^{\oplus 4} \rightarrow \A_{g,r} \rightarrow (A-I)\backslash\A_{g,r} \rightarrow 0~
\end{equation*}
which is a free resolution of the transfer bimodule $(A-I)\backslash\A_{g,r}$ as a right module. It is in fact an $\A'\mathrm{-}\A_{g,r}$-bimodule resolution where the left action by $\A'$ is twisted as it braids past the free basis. For instance, let $\tilde{b}^i_j$ denote generators from some copy $\OqG$ in $\A'$. Then using the PBW relations of \eqref{eq:braided-Oq-Oq} we deduce the following action of $\tilde{b}^1_1$ on $\bigoplus_{i=1}^{4}\widehat{x}_i\A_{g,r}$
\begin{align*}
&\tilde{b}^1_1 . \widehat{x}_1 a = \widehat{x}_1 \tilde{b}^1_1 a  - \numberq{-2}{q} \widehat{x}_4 \tilde{b}^1_2a \\
& \tilde{b}^1_1 . \widehat{x}_2 a = \widehat{x}_2 \tilde{b}^1_1 a  + \numberq{-2}{q} \widehat{x}_4 \tilde{b}^1_2a \\
&\tilde{b}^1_1 . \widehat{x}_3 a = \widehat{x}_3 \tilde{b}^1_1 a  - \numberq{-2}{q} \widehat{x}_2 \tilde{b}^1_2a + \numberq{-2}{q}\widehat{x}_1 \tilde{b}^1_2 a \\
&\tilde{b}^1_1 . \widehat{x}_4 a = \widehat{x}_4 \tilde{b}^1_1 a ~.
\end{align*}
Similarly, one has we relations for any other generators and degrees of the complex. 

In particular, we have determined a Koszul resolution for the transfer bimodules from \eqref{eq:ns-transfer} and \eqref{eq:s-transfer}. 
Thus, the derived inverse image $Li^\ast M$ resp.\ $L\delta(M)$ of an $\A_{g,r}$-module $M$ is represented by the complex: 
\begin{equation*}
	0 \rightarrow M \rightarrow M^{\oplus 4} \rightarrow M^{\oplus 6} \rightarrow M^{\oplus 4} \rightarrow M 
\end{equation*}
We can also define the transfer bimodule $\Aplus_{g-1\rightarrow g}$ and write $Li^\ast_+ := \Aplus_{g-1\rightarrow g}\otimes^{\mathbb{L}} -$. In particular, 
if $M$ is a finitely generated $\A_g$-module with an $\Aplus_g$-lattice $M'$, we have $Li^\ast M \cong (Li^\ast_+M')[\detq^{-1}]$ and similarly for the separating case $L\delta_+$. 

The $\SL_2$-case is slightly easier due to the $q$-determinant relation. Indeed, the (right) $\OqSL{2}$-module $\field$ is isomorphic to 
\[
\field := (A-I)\backslash \OqSL{2} \cong (a^1_2,a^2_1,a^2_2-1)\backslash\OqSL{2}
\]
where in the last equation the generator $a^1_1-1$ was disregarded due to the $q$-determinant relation $a^1_1a^2_2 -q^2 a^1_2 a^2_1  =1 $. We obtain the following Koszul resolution of $\field$ 
\begin{equation}
    0 \rightarrow \OqSL{2}\rightarrow \OqSL{2}^{\oplus 3}\rightarrow \OqSL{2}^{\oplus3} \rightarrow \OqSL{2}\label{eq:Koszul-SL2}
\end{equation}
with differentials
\begin{align}\label{eq:Koszul-differentials-SL2}
    \delta(\widehat{x}_1) &= a^2_2 -1 \nonumber\\ 
    \delta(\widehat{x}_2) &= a^1_2  \nonumber\\
    \delta(\widehat{x}_3) &= a^2_1  \nonumber\\
    \delta(\widehat{r}_1) &=  \widehat{x}_1 a^1_2 - \widehat{x}_2 q^{2} (a^2_2 -q^{-2}) \nonumber\\
    \delta(\widehat{r}_2) &=  \widehat{x}_1 a^2_1 - \widehat{x}_3 q^{-2} (a^2_2 -q^{2}) \nonumber\\
    \delta(\widehat{r}_3) &=  -\widehat{x}_2 q^2 a^2_1  + \widehat{x}_3 a^1_2 - \widehat{x}_1 \numberq{-2}{q}(a^2_2 +1)\nonumber\\
    \delta(\widehat{c}) &= -\widehat{r}_1 q^2 a^2_1 + \widehat{r}_2 a^1_2 + \widehat{r}_3 (a^2_2-1)~
\end{align}
where $\{\widehat{c}\}$, $\{\widehat{r}_1,\widehat{r}_2,\widehat{r}_3\}$ and $\{\widehat{x}_1,\widehat{x}_2,\widehat{x}_3\}$ denote an $\OqSL{2}$ basis of the first, second, and third terms (reading from the left) in the Koszul resolution \eqref{eq:Koszul-SL2}. 

In particular, if $M$ is a $\A_g$-module for $G= \SL_2$ its inverse image  $Li^\ast M$ is represented by the complex
\[0 \rightarrow M \rightarrow M^{\oplus3} \rightarrow M^{\oplus 3} \rightarrow M \rightarrow 0~.\]

\subsection{Ore localisations}\label{subsec:Ore-Loc}

Associated to closed immersions $G^{g+r-2} \hookrightarrow G^{g+r-1}$ of codimension $m =\dim G$, we consider the associated open immersion $j:U\hookrightarrow G^{g+r-1}$. In this section, we will construct localisations along $U$ of $\A_{g,r}$-modules 
\begin{equation*}
	\phi: M \rightarrow j_\ast M_U~. 
\end{equation*} 
The notation is motivated by algebraic geometry and the theory of $D$-modules. However, as $U$ is not affine, we first find an affine cover $\{U_i\}_{i=1,\dots, m}$. 

\begin{remark}
	The technical obstruction in general will be to find Ore sets $S_{i} \subset \A_{g,r}$ such that $S^{-1}_{i}M$ is a deformation quantization of the classical (commutative) localization $(M_{q=1})|_{U_i}$.  
\end{remark}

The two types of immersions correspond to the previously defined $Li^\ast$ and $L\delta$. For both cases, we will write $A=(a^i_j)$ for the generators of the corresponding cycle $\OqG$. The generators $D=(\partial^i_j)$ will denote the complementary generators in $\DqG$ for the non-separating case and $F=(f^i_j)$ denote the complementary generators in $\DqFRT$ in the separating case. Generators of other tensorants in the braided tensor product will be collectively denoted by $\tilde{b}^i_j$ if they come from an $\OqG$-factor and $\tilde{f}^i_j$ if they come from an $\OqFRT$-factor.

\begin{example}
For $G = \GL_2$ we will take the affine cover with $U_1 = \{\det(A)-1 \neq 0\}, U_2= \{a^1_2 \neq 0\}, U_3 = \{a^2_1\neq 0\}$ and $U_4 = \{a_2^2-1\neq 0\}$. Note that we could have taken $U_1' = \{a_1^1-1\neq 0\}$ but we find the Ore set for $U_1$ to be more convenient.

For $G=\SL_2$ we only take $U_2,U_3$ and $U_4$ as above due to the $\detq(A)=1$ relation.
\end{example}

Then, the localization along $U$ should be described by the corresponding \v{C}ech complex: 
\begin{equation*}
	 0 \rightarrow j_\ast M_U \rightarrow \mathfrak{C}^1(M,U)\rightarrow \dots \rightarrow \mathfrak{C}^m(M, U )\rightarrow 0 
\end{equation*} 
where 
\begin{equation*}
\mathfrak{C}^k(M,U) := \bigoplus_{i_1,\dots ,i_k}(j_{i_1\dots i_k})_\ast M_{U_{i_1 \dots i_k}}~.
\end{equation*}
We now give a full definition for $(j_{i_1\dots i_k})_\ast M_{U_{i_1 \dots i_k}}$ as localization along certain Ore sets $S_{i_1\dots i_k} \subset \A_g$. 

Since we deal with localisation in non-commutative algebras we care about Ore sets, which we now briefly recall: 
\begin{definition}
	A subset $S$ in an algebra $A$ is called a \textbf{(left) Ore set} if it is: 
	\begin{enumerate}
		\item multiplicatively closed and
		\item for any pair $(s,a) \in S\times A$ the intersection $S\cdot a \cap A\cdot s$ is non-empty. 
	\end{enumerate}
\end{definition}
The second condition in the above definition is referred to as the \textbf{(left) Ore condition}. Explicitly, for some $s \in S$ and $a\in A$ there exist $\tilde{s}\in S$ and $\tilde{a}\in A$ such that
\begin{equation*}
	\tilde{s} \cdot a = \tilde{a} \cdot s~.
\end{equation*}

\begin{lemma}\label{lem:Ore-from-generators}
Let $A$ be a finitely generated algebra with generators $V = \{x_1,\dots , x_n\} \subset A$. Let $S\subset A$ be a commutative multiplicatively closed subset satisfying the Ore condition with respect to the generators, \ie 
	\begin{equation*}
		S\cdot x_i \cap A\cdot s \neq \emptyset \quad \forall (s,x_i)\in S\times V \subset S \times A
	\end{equation*}
	Then, $S$ is an Ore set.
\end{lemma}
\begin{proof}
	We first prove that the Ore condition holds for all monomials in $V$. Consider a monomial $m=x_{i_k}\cdots x_{i_1}$ in $V$ and a fixed element $s\in S$. Then, by assumption there exist $s_1\in S$ and $a_1\in A$ such that $s_1 \cdot x_{i_1} = a_1 \cdot s$ and continuing by induction we also find $s_j\in S$ and $a_j$ for $j\in\{2,\dots,k\}$ such that $s_j\cdot x_{i_j} = a_j \cdot s_{j-1}$. Thus, we found the following Ore solution 
        \[s_k \cdot x_{i_k}\cdots x_{i_{1}} = a_k\cdots a_1 \cdot s~.\]
	
	We now show that the Ore condition is satisfied with respect to all elements in $A$. Any $a \in A$ can be written as a linear combination of monomials $m_i$ in $V$ 
	\begin{equation*}
		a = \sum_{i\in I}{\lambda_i ~ m_i}
	\end{equation*}
	By the first part of the proof, for any $s \in S$ there exist $s_i \in S$ and $\tilde{a}_i\in A$ such that 
	\[s_i \cdot m_i = \tilde{a}_i \cdot s~.\]
	In particular,
	\begin{equation*}
		\left(\prod_{i\in I}{s_i}\right) \cdot a= \sum_{i\in I}{\lambda_i \left(\prod_{j\neq i}{s_j}\right)\cdot s_i \cdot m_i} = \left(\sum_{i\in I}{\lambda_i \prod_{j\neq i}{s_j} \tilde{a}_i}\right) s
	\end{equation*}
    concluding our proof. 
\end{proof}
Given an Ore set $S \subset A $ one defines the \textit{Ore localisation} $S^{-1}A$ and for an $A$-module $M$ we define the $S^{-1}A$-module 
\[S^{-1}M:= S^{-1}A\otimes_{A} M\]
using the $S^{-1}A\mbox{-}A$-bimodule structure of $S^{-1}A$. The canonical map $\phi_S: M \rightarrow S^{-1}M, m \mapsto 1\otimes m$ is an $A$-module map and its kernel defines the $S$-\textit{torsion submodule} of $M$
\[t_{S}(M):= \ker(\phi_S) = \{m \in M\mid \exists s \in S: s.m =0\}~.\]
Note that the Ore condition is used to show that $t_S(M)$ is a submodule \cite[Lem.\ 4.21]{Goodearl-Warfield}. We say that $M$ is $S$-\textit{torsion} if $t_{S}(M) = M$ and it is $S$-\textit{torsion-free} if $t_S(M) = 0$. 

We now return to our affine cover $U_i$ and find appropriate Ore sets. The following proposition holds for any $q$, \ie over $\Cq$, but throughout this section we work with $q$ generic. 

\begin{proposition}\label{prop:Ore-sets}
	The sets $S_i$ monomially generated by: 
	\begin{equation*}
		T_1 = \left\{\detq(A) - q^{2k}\right\}_{k\in \mathbb{Z}}, \quad T_2 = \left\{a_{2}^1\right\}, \quad T_3 = \left\{a_{1}^2\right\}, \quad T_4 = \left\{(a_2^2 - q^{2k})\right\}_{k\in K}
	\end{equation*}
	form Ore sets in $\Aplus_{g,r}$, where $K = \mathbb{Z}$ for $\GL_2$ and $K= \frac{1}{2}\mathbb{Z}$ for $\SL_2$. In particular, they form Ore sets in $\A_{g,r}$.
\end{proposition}
Hence, for an $\A_{g,r}$-module $M$ we have localisations 
\[\phi_i: M \rightarrow S^{-1}_i M\]
and $M$ is (set-theoretically) $q$-\textit{supported} on $Z_i$ if it is $S_i$-torsion, while it is $q$-\textit{supported} on $U_i$ if it is $S_i$-torsion-free. 

Furthermore, we will show that such Ore localisations preserve holonomicity in the following sense. If $M$ is a holonomic over $\A_{g,r}$, then $S^{-1}_2 M $ and $S^{-1}_3 M$ are holonomic $\A_{g,r}$-modules while $S^{-1}_1 M$ and $S^{-1}_4 M$ are filtered by holonomic submodules. Note in passing that the Ore set $S_1$ is mostly relevant for $G=\GL_2$ since the $S_1$-localisation vanishes for $\SL_2$. 

The filtered submodules of $S^{-1}_i M$ for $ i=1,4$ are defined as follows. For $n \in \mathbb{Z}_{\geq 0}$ and $i=1,4$ define the subsets 
\begin{equation*}
    S_{i}^{(n)} := \left\{\prod_{j\in \mathbb{Z}}t_k^{\gamma_j} \mid \gamma_{k}\leq n, t_k \in T_i\right\}\subset S_i
\end{equation*}
for the generators $t_j\in T_i$ of $S_i$. For $\varepsilon=1$ one also includes half-integer generators $\{t_k\}_{k\in \frac{1}{2}\mathbb{Z}}$. These form an $\mathbb{N}$-filtration of $S_i$ as a monoid. Define the associated submodules
\begin{equation*}
    S_i^{(n)}M:= \left\{ m \in S^{-1}_i M \mid \exists s\in S_i^{(n)}: s.m \in \phi_i(M)\right\}\subset S^{-1}_i M~. 
\end{equation*}

We split the proof of Proposition~\ref{prop:Ore-sets} in four lemmata and we describe the $\Aplus_{g,r}$-action on $S_{i}^{-1} M$. Holonomicity preservation will be proven by showing $\GKdim(S^{-1}_iM) \leq \GKdim(M)$. 

\begin{lemma}\label{lem:s1}
    The set $S_1$ forms an Ore set in $\Aplus_{g,r}$ and $S^{(n)}_{1}(-)$ for $n\geq1$ preserves holonomicity. 
\end{lemma}
\begin{proof}
	The $\SL_2$-case is redundant here and thus we only consider $G=\GL_2$. 
    Recall that the $q$-determinant $\det_q(A)$ commutes with all $\tilde{b}^i_j,\tilde{f}^{i}_j$-generators and $q^{2}$-commutes with $\partial^i_j$-generators in the non-separating case \eqref{eq:detq-commutation} resp.\ $f^i_j$ in the separating case \eqref{eq:detqF-commutation}.
    Hence, the only non-trivial Ore condition for the generators $d_k = \detq{A} -q^{2k}$ of $S_1$ is realised by
    \begin{equation*}
        d_{k+1}\partial^i_j = q^2~\partial^{i}_j d_k\quad \text{resp.} \quad d_{k-1}f^i_j = q^{-2}~f^{i}_j d_k~.
    \end{equation*}
    The Ore conditions for an arbitrary element $s\in S_1$ with respect to the generators follow easily from the above equation and thus by Lemma~\ref{lem:Ore-from-generators} $S_1$ is an Ore set.
    
    Now, let $M$ be an $\Aplus_{g,r}$-module with a good filtration $F_l M$ with respect to the standard filtration of $\Aplus_{g,r}$. Define the following filtration of $S^{(1)}_{1}M$ by 
    \begin{equation*}
        F_{l} S^{(1)}_1 M := \left(\prod_{k=-l}^l{d_k}\right)^{-1}\otimes F_{3l} M~.
    \end{equation*}
    For instance, to verify that it is a well-defined filtration note (in the non-separating case) that
    \begin{align*}
        \partial^{i}_j. \left(\left(\prod_{k=-l}^l{d_k}\right)^{-1}\otimes F_{3l} M\right) &= \left(\prod_{k=-l}^l{d_{k+1}}\right)^{-1}\otimes q^{2(2l+1)} \partial^i_j F_{3l} M \\ 
        &= \left(\prod_{k=-l-1}^{l+1}{d_k}\right)^{-1}\otimes q^{2(2l+1)} d_{-l-1}\, d_l \,\partial^i_j F_{3l} M\\
        &\subset \left(\prod_{k=-l-1}^{l+1}{d_k}\right)^{-1}\otimes F_{3l+3} M =: F_{l+1}S_1^{(1)}M
    \end{align*}
	and similarly for the separating case. In particular, \[\dim(F_l S_1^{(1)}M) \leq \dim(F_{3l}M)\]
    which implies $\GKdim(S^{(1)}_1 M) \leq \GKdim(M)$. By Proposition~\ref{prop:Agr-holonomic} and the above argument $S^{(1)}_1$ preserves holonomicity. The same is true for $S^{(n)}_1$ as $S^{(n)}_1 = S_1^{(1)}\circ \cdots \circ S_{1}^{(1)}$. 
\end{proof}

\begin{lemma}\label{lem:s2}
	The set $S_2$ generated by the element $a_2^1$ forms an Ore set and $S_2^{-1}$ preserves holonomicity. 
\end{lemma}
\begin{proof}
	We list the Ore conditions with respect to the generators of $\Aplus_{g,r}$ that are not immediate from the relations 
	\begin{align*}
		& \left(a_2^1\right)^2 a_1^2 = \left(a_2^1 a_1^2 + q^{-2}\numberq{-2}{q}\left(a_1^1 - q^{-4}a_2^2\right) a_2^2\right) a_2^1\\ 
		& \left(a_2^1\right)^2 \partial_1^1 = q^{2-\varepsilon}\left(a_2^1 \partial_1^1 - q^{-2}\numberq{-2}{q}a_2^2\partial_2^1 -(\numberq{-2}{q})^2 a_2^1 \partial_2^2\right) a_2^1 \\
		& \left(a_2^1\right)^2 \partial_1^2 = q^{-\varepsilon}\left(a_2^1 \partial_1^2 - q^{-4}\numberq{-2}{q}a_2^2 \partial_2^2\right) a_2^1~,
	\end{align*}
	for the non-separating case. For the separating case, the last two relations are replaced by 
	\[
		\left(a^1_2\right)^2 f^2_j = q^{\varepsilon-1} \left( a^1_2 f^2_j + q^{-2} \numberq{-2}{q} a^2_2  f^1_j\right) a^1_2		
	\]
	The non-trivial Ore conditions  with respect to the $\tilde{b}^i_j$ generators are 
        \begin{align*}
	\left(a^1_2\right)^2\tilde{b}^1_1 &=\left(a^1_2  \tilde{b}^1_1 + q^{-2} \numberq{-2}{q} \left(a^1_1 - q^{-4} a^2_2\right) \tilde{b}^1_2\right) a^1_2\\ 
    a^1_2 \tilde{b}^1_2 &= q^{-2} \tilde{b}^1_2 a^1_2 \\ 
    \left(a^1_2\right)^3 \tilde{b}^2_1 &= q^2 \bigg(\left(a^1_2\right)^2  
       \tilde{b}^2_1 -\numberq{-2}{q} \left(a^1_1 + 
         \lambda(q) a^2_2\right)\left(a^1_2 \left(\tilde{b}^1_1 - 
             \tilde{b}^2_2\right) + \numberq{-4}{q} \left(a^1_1 - q^{-4} a^2_2\right) \tilde{b}^1_2\right)\\ \nonumber
             &\hspace{1cm}+ \mu(q) \left(a^1_2 a^2_1 + 
          q^{-2} \numberq{-2}{q} \left(a^1_1 - q^{-4} a^2_2\right)  a^2_2\right)  \tilde{b}^1_2 \bigg)  
    a^1_2 \\
\left(a^1_2\right)^2 \tilde{b}^2_2 &=(a^1_2  \tilde{b}^2_2 + q^{-2} \numberq{-2}{q} (a^1_1 - q^{-4} a^2_2) \tilde{b}^1_2)  
    a^1_2
	\end{align*}
    where $\lambda(q) =  q^{-2} \left(1 - q^{-2} - q^{-4} \right)$ and $\left(q^{-2} - q^{-6} + q^{-4}- 1\right)$. The only non-trivial Ore condition with respect to the $\tilde{f}^i_j$ generators is 
    \[
    \left(a^1_2\right)^2 \tilde{f}^2_j = q\left(a^1_2 \tilde{f}^2_j - \numberq{-2}{q}\left(a^1_1 - q^{-4}a^2_2\right) \tilde{f}^1_j\right)a^1_2~.
    \]
    
    Given a generator $x\in \Aplus_{g,r}$ let $o(x)\in\mathbb{Z}_{n\geq0}$ the minimal number such that $(a^1_2)^{o(x)}x = y a^1_2$ for some $y\in \Aplus_{g,r}$. For instance, $o(a^2_2) = 0$, $o(a^2_1) = 1$ and $o(\tilde{b}^2_1)=2$ as read from the previous equations. In fact, $o(x)=2$ is the maximum possible value for any generator $x$. 
    
    Using the relations, one can inductively show that for any $s=(a^1_2)^n$ and a generator $x$ of $\Aplus_{g,r}$ there exist $\tilde{s} = (a^1_2)^{n+o(x)}$ and $y\in \Aplus_{g,r}$ of order $o(x)+1$ satisfying the Ore condition $(a^1_2)^{n+o(x)}x = y (a^1_2)^n$. This follows inductively by multiplying $a^1_2$ on the left of the above relations. This is easily observed as follows: if $(a^1_2)^{1+o(x)}x = y a^1_2$ as above, then the terms of $y$ appearing all satisfy the Ore condition by multiplying a single $a^1_2$ on the left. 
	
    Thus, the action of a generator $x$ on $S^{-1}_2 M$ is given by 
	\begin{equation*}
		x . \left((a_2^1)^{-n}\otimes m\right) := \left(a_2^1\right)^{-n-o(x)}\otimes y.m
	\end{equation*} 
	and if $F_l M$ is a filtration of $M$ then 
	\begin{equation*}
		F_l (S^{-1}_2 M) = (a_2^1)^{-2l}\otimes F_{3l} M~.
	\end{equation*}
	is a filtration of $S_2^{-1}M$. This is a well-defined filtration since $o(x)\leq 2$ and $y$ of order $\leq3$. In particular, we get $\dim(F_lS^{-1}_2M)\leq \dim(F_{3l}M)$ and we repeat the argument as before. 
\end{proof} 
\begin{lemma}\label{lem:s3}
	The set $S_3$ generated by the element $a_1^2$ forms an Ore set and $S_3^{-1}$ preserves holonomicity.
\end{lemma}
\begin{proof}
We list the Ore conditions that are not immediate from the relations. The equations involving $\partial^i_j$ are for the non-separating case and those involving $f^i_j$ for the separating case. 
    \begin{align*}
\left(a^2_1\right)^2 a^1_2 &= \left(a^2_1 a^1_2 + \numberq{2}{q} a^2_2 \left(a^1_1 - \lambda(q) a^2_2\right)\right) a^2_1 \\ 
\left(a^2_1\right)^2 \partial^1_1 &= q^{-\varepsilon}\left(a^2_1 \partial^1_1 +  \numberq{2}{q} \left(a^1_1 - \lambda(q) a^2_2\right) \partial^2_1\right) a^2_1 \\ 
\left(a^2_1\right)^3 \partial^1_2 &= q^{-\varepsilon}\bigg(\left(a^2_1\right)^2 \partial^1_2  -\numberq{-2}{q} \left(a^1_1 + q^2 \numberq{-4}{q} a^2_2\right) \left(q^2 a^2_1 \partial^2_2 + q^4 \numberq{2}{q} a^2_2 \partial^2_1 \right)\\ 
&\hspace{1cm}- \numberq{2}{q}^2 \left(a^2_1 a^1_2 + \numberq{2}{q} a^2_2 \left( a^1_1 - \lambda(q) a^2_2\right)\right) \partial^2_1 \nonumber \\ 
&\hspace{1cm}+ q^2 \numberq{2}{q} a^2_2 \left(a^2_1 \partial^1_1 + \numberq{2}{q} \left(a^1_1 - \mu(q) a^2_2\right)  \partial^2_1 - q^2 a^2_1 \partial^2_2 \right) \bigg) a^2_1 \nonumber \\
\left(a^2_1\right)^2 \partial^2_2 &= q^{-\varepsilon}\left(q^2 a^2_1 \partial^2_2 + q^4 \numberq{2}{q} a^2_2 \partial^2_1 \right) a^2_1  \\
(a^2_1)^2 f^1_j &= q^{-1+\varepsilon} (a^2_1 f^1_j -\numberq{2}{q} a^2_2 f^2_j) a^2_1
\end{align*}
where $\lambda(q)=\left(1 -q^{-2} + q^2\right)$ and $\mu(q)=\left(1 - q^{-2} + q^2 + q^4\right)$. Note that the Ore conditions with respect to generators $\tilde{b}^i_j$ and $\tilde{f}^i_j$ of other tensorants are immediate from the relations. 
In particular, as in the previous proof, for each generator $x$ there exists $y\in (\Aplus_{g,r})_{\leq 3}$ such that
	\begin{equation*}
		x. \left((a^2_1)^{-n}\otimes m\right) := \left(a^2_1\right)^{-n-2}\otimes y.m~.
	\end{equation*} 
	Therefore, we define a filtration 
	\begin{equation*}
		F_l (S^{-1}_3 M) = (a^2_1)^{-2l}\otimes F_{3l} M~.
	\end{equation*}
\end{proof}
Finally, we treat the $S_4$-case: 
\begin{lemma}\label{lem:s4}
The set $S_4$ monomial generated by 
\begin{equation*}
	\{t_j:= a_2^2 - q^{2k}\}_{k\in K}
\end{equation*}
forms an Ore set in $\Aplus_{g,r}$ and $S^{(1)}_i(-)$ preserves holonomicity. The index set $K$ is $\mathbb{Z}$ for $\GL_2$ and $\SL_2$.
\end{lemma}
\begin{proof}
	We first demonstrate the Ore conditions with respect to the $\partial^i_j$'s resp.\ $f^i_j$'s. From the algebra relations we have  
		\begin{align*}
		&t_k a_1^1  = a_1^1 t_k \\
		&t_k a^1_2  = q^2 a^1_2 t_{k-1} \\
		&t_k a^2_1  = q^{-2} a^2_1 t_{k+1} \\
		&q^{\varepsilon}t_{k-\frac{\varepsilon}{2}} \partial_1^2  = \partial_1^2 t_{k}\\ 
		&q^{\varepsilon}t_{k-\frac{\varepsilon}{2}} \partial_2^2  = q^{2} \partial_2^2 t_{k-1}\\
		&q^\varepsilon t_{k-\frac{\varepsilon}{2}} \partial_1^1 = \partial_1^1 t_{k} +q^\varepsilon \numberq{2}{q}a_2^1 \partial_1^2 \\ 
		&q^\varepsilon t_{k-\frac{\varepsilon}{2}} \partial_2^1 =\partial_2^1 t_{k-1} -q^{2+\varepsilon}\numberq{2}{q}a_2^1 \partial_2^2\\ 
        &q^{-\varepsilon}t_{k+\frac{\varepsilon}{2}} f^1_l = f^1_l t_{k} \\ 
        &q^{-\varepsilon}t_{k+\frac{\varepsilon}{2}} f^2_l =q^{-2}f^2_l t_{k+1}  
	\end{align*} 
	The Ore conditions are immediate for all generators but $\partial_1^1$ and $\partial_2^1$. These are realised by the following equations: 
	\begin{align*}
		&q^{2\varepsilon}\left(t_{k-\frac{\varepsilon}{2}} t_{k+1-\frac{\varepsilon}{2}}\right)\partial_1^1 = \left(\partial_1^1 t_{k+1} + q^{\varepsilon} \numberq{4}{q}a_2^1 \partial_1^2\right)t_{k}\\ 
		&q^{\varepsilon}\left(t_{k+1-\frac{\varepsilon}{2}} t_{k+2-\frac{\varepsilon}{2}}\right)\partial_2^1 = q^2\left(q^2\partial_2^1 t_{k+1} + q^\varepsilon\numberq{4}{q}a_2^1 \partial_2^2\right)t_k
	\end{align*}
From the above mentioned relations it is immediate that the Ore condition is satisfied for any $s\in S_4$ with respect to the $a^i_j$'s, $\partial^i_j$'s resp.\ $f^i_j$'s. 
However, we write the following relation which will be used later to define an appropriate filtration:
	\begin{align*}
		&\left(t_{k-\frac{\varepsilon}{2}} \cdots t_{k+N+1-\frac{\varepsilon}{2}}\right) \partial_1^1 = \left(\partial_1^1 t_{j+N+1} + q^\varepsilon\numberq{2(N+1)}{q}a_2^1 \partial_1^2\right) \left(t_k \cdots t_{k+N}\right)\\
		&\left(t_{k+1-\frac{\varepsilon}{2}} \cdots t_{k+N+2-\frac{\varepsilon}{2}}\right) \partial_2^1 = q^{2(N+1)}\left(\partial_2^1 t_{k+N} - q^\varepsilon\numberq{-2(N+1)}{q}a_2^1 \partial_2^2\right) \left(t_k \cdots t_{k+N}\right)
	\end{align*}
The Ore conditions with respect to the generators of the other tensorants are realised by:
\begin{align*}
t_{k-1} t_k \tilde{b}^1_1 &= \left(t_{k-1} \tilde{b}^1_1 + q^{-2} \numberq{-2}{q} a^2_1 \tilde{b}^1_2 \right) t_k  \\ 
t_k \tilde{b}^1_2 &= \tilde{b}^1_2 t_k \\ 
t_{k-2} t_{k-1} t_k \tilde{b}^2_1 &=
\left(t_{k-2} t_{k-1} \tilde{b}^2_1 - q^{-2}\numberq{-2}{q} a^2_1 t_{k-1} \left(\tilde{b}^1_1 - \tilde{b}^2_2\right) + q^{-4} \lambda(q) \left(a^2_1\right)^2 \tilde{b}^1_2\right) t_k \\ 
t_{k-1} t_k \tilde{b}^2_2  &=\left(t_{k-1} \tilde{b}^2_2 -\numberq{-2}{q} a^2_1 \tilde{b}^1_2\right) t_k\\ 
t_k \tilde{f}^1_j &= \tilde{f}^1_j t_k\\ 
t_{k-1}t_k \tilde{f}^2_j &= \left(t_{k-1} \tilde{f}^2_j - \numberq{-2}{q}a^2_1 \tilde{f}^1_j\right) t_k
\end{align*}
where $\lambda(q) = \left(1 - q^{-4} + q^{-2} - q^2 \right)$.
In particular, for $\varepsilon=0$ for any PBW generator $x$ there exists $y$ of order at most $6$ such that: 
\begin{equation*}
	x .\left(\left(t_k \cdots t_{k+N}\right)^{-1} \otimes m\right) = \left(t_{k-2}\cdots t_{k+N+2}\right)^{-1}\otimes y. m
\end{equation*}
For $\varepsilon=1$, one has the same expression by including half-integer steps in the product, namely
\[x.\left(\left(\prod_{j=0}^{N}t_{k+\frac{j}{2}}\right)^{-1}\otimes m\right) = \left(\prod_{j=-4}^{N+4}t_{k+\frac{j}{2}}\right)^{-1}\otimes y.m\]
In particular, the submodule $S_4^{(1)}$ preserves holonomicity, since for a given filtered module $(M,F)$ we find the following filtration of $S^{(1)}_4 M$ and repeat the growth arguments 
\begin{equation*}
	F_l S_4^{(1)} M = \left(\prod_{k=-2l}^{2l}{t_k}\right)^{-1}\otimes F_{6l}M~.
\end{equation*}
\end{proof}

\begin{lemma}\label{lem:Inv-Im-Loc}
	Let $M$ be an $\Aplus_{g,r}$-module. Then, 
	\begin{equation*}
		Li^\ast_+(S^{-1}_i M) \simeq 0
	\end{equation*} 
	resp.\
	\[L\delta_+(S^{-1}_i M) \simeq 0\]
	for any of the Ore sets $S_i$. 
\end{lemma}
\begin{proof}
	It suffices to check that $S^{-1}_i \Aplus_{g-1\rightarrow g} \cong 0$ as
	\[Li^\ast_+ (S^{-1}_i M) \cong \Aplus_{g-1\rightarrow g}\otimes^{\mathbb{L}}_{\Aplus_{g,r}} S^{-1}_i M \cong S^{-1}_i \Aplus_{g-1\rightarrow g} \otimes^{\mathbb{L}}_{S^{-1}_i\Aplus_{g,r}} S^{-1}_i M\cong 0~.\]
	This follows directly, since by definition $\Aplus_{g,r}$ is $S_i$-torsion. 
    Similarly, for the separating case and $L\delta_+$.
\end{proof}

\subsubsection*{Intersection of Ore sets}

The product $S\cdot S' = \{s\cdot s' \mid s\in S, s'\in S'\}$ of two Ore sets $S,S'$ satisfies the Ore conditions, but it is not multiplicatively closed in general. Nevertheless, the multiplicative closure $\langle S\cdot S'\rangle$ of $S\cdot S'$ is an Ore set \cite[Prop. 6.9]{Skoda}. Let $S_i$ be the four Ore sets considered for the affine cover of $A\neq \unit$. For any ordered tuple $(i_1 < \dots < i_k)$ of integers $i_1,\dots, i_k \in \{1,\dots,4\}$ define the multiplicative closures 
\begin{equation*}
S_{i_1\dots i_k} := \langle S_{i_1} \cdots S_{i_k}\rangle ~.
\end{equation*}
These will be the Ore sets corresponding to the intersections $U_{i_1\dots i_k}= \bigcap_{j=1}^k U_{i_j}$. 
\begin{remark}
In fact, $S_1$ and $S_4$ behave well with respect to the other Ore sets in that the images of $S_1$ and $S_4$ under $\phi_i: A\rightarrow S^{-1}_i A$ are still Ore sets in $S^{-1}_i A$. In particular 
\[S^{-1}_{1i}A \cong \phi_i(S_1)^{-1} S^{-1}_i A\]
and similarly for $S_4$. 
\end{remark}

\subsection{Kashiwara's theorem}

Let $M$ be a finitely generated $\A_{g,r}$-module which is $q$-supported in $\{A=I\}\subset G^{2(g+r-1)}$, \ie 
\[M = t_{S_1}^{(1)}(M) \cap t_{S_2}(M)\cap t_{S_3}(M)\cap t_{S_4}^{(1)}(M)~.\]

Recall that the $q$-support condition is the $q$-analogue of set-theoretic support.  The stronger notion, of scheme-theoretic support is captured by
\[M_0 := \ker(\detq(A)-1)\cap \ker(a^1_2)\cap \ker(a^2_1)\cap \ker(a^2_2 -1)~\subset M,\]
with its natural $\A'$-module structure, where $\A'$ is $\A_{g-1,r}$ in the non-separating case and $\A_{g_1,r_1}\boxtimes \A_{g_2,r_2}$ in the separating case.

\begin{remark}
We note that for $G= \SL_2$ and any $\A_{g,r}$-module $M$, we already have $t^{(1)}_{S_1}(M)=\ker(\detq(A)-1)=M$ as by definition it is annihilated by $\detq(A) -1$.
\end{remark}

Since $q$ is generic we have a $\mathbb{Z}^2$-grading
\begin{equation*}
    M = \bigoplus_{(j,j') \in \mathbb{Z}^2}M_{j,j'}
\end{equation*}
for $G=\GL_2$, where $M_{j,j'} := \ker(d_j)\cap \ker(t_{j'})$ for $d_j := \detq(A) -q^{2j}$ and $t_{j'}:= a^2_2 - q^{2j'}$. 
Similarly, we have 
\[M = \bigoplus_{j\in \frac{1}{2}\mathbb{Z}} M_{1,j}\]
for $\SL_2$. Moreover, there is an $\mathbb{N}^2$-filtration (with the product order) of $M$ by 
\begin{equation*}
    M^{(k,l)} := \ker(a^1_2)^k \cap \ker(a^2_1)^l\subset M
\end{equation*}
In this notation we have 
\begin{equation*}
M_0 := M_{0,0}^{(1,1)} = \ker(d_0)\cap \ker(a^1_2) \cap \ker(a^2_1) \cap \ker(t_0) \subset M~.
\end{equation*}
It is immediate from the Ore conditions that for $\GL_2$ (similarly for $\SL_2$) 
\begin{align*}
    &a^1_1: M^{(k,l)}_{j,j'} \rightarrow M^{(k,l)}_{j,j'} &\\
    &a^1_2: M^{(k,l)}_{j,j'} \rightarrow M^{(k-1,l+1)}_{j,j'+1} &\\
    &a^2_1: M^{(k,l)}_{j,j'} \rightarrow M^{(k+1,l-1)}_{j,j'-1} &\\
    &a^2_2: M^{(k,l)}_{j,j'} \rightarrow M^{(k,l)}_{j,j'}
\end{align*}
and
\begin{align*}
    &\textit{(non-separating)} && \textit{(separating)}\\ 
    &\partial^1_1: M^{(k,l)}_{j,j'} \rightarrow M^{(k+1,l+1)}_{j+1,j'} \oplus M^{(k+1,l+1)}_{j+1,j'+1} &&f^1_1,f^1_2: M^{(k,l)}_{j,j'} \rightarrow M^{(k,l+1)}_{j-1,j'}  \\
    &\partial^1_2: M^{(k,l)}_{j,j'} \rightarrow M^{(k,l+2)}_{j+1,j'+1} \oplus M^{(k,l+2)}_{j+1,j'+2} && f^2_1,f^2_2: M^{(k,l)}_{j,j'} \rightarrow M^{(k+1,l)}_{j-1,j'-1} \\
    &\partial^2_1: M^{(k,l)}_{j,j'} \rightarrow M^{(k+1,l)}_{j+1,j'} && \\
    &\partial^2_2: M^{(k,l)}_{j,j'} \rightarrow M^{(k,l+1)}_{j+1,j'+1} &&
\end{align*}

Recall that Kashiwara's theorem gives an equivalence between the subcategory of $\D$-modules on a smooth variety $X$ which are set-theoretically supported on a closed subvariety $Z$, and the category of $\D(Z)$-modules (see \cite[Thm.\ 1.6.1]{Hotta}).  In particular, denoting the inclusion by $i:Z\to X$, Kashiwara's theorem shows that a $\D(X)$-module $M$ set-theoretically supported on $Z$ agrees with the pushforward $i_*M_0$, where $M_0=i^*M[n]$ is the pullback of $M$ shifted by the codimension $n=\operatorname{codim}Z$.

The main result of this section is that the $\A'$-submodule $M_0$ of the holonomic $\A_{g,r}$module $M$ is also holonomic.  It is proved by establishing injectivity of the quantum analogue of the map $i_*i^*M[n]\to M$ as in Kashiwara's theorem:
\begin{proposition}\label{prop:kernel}
    If $M$ is a holonomic $\A_{g,r}$-module, then $M_0$ is holonomic as an $\A'$-module. 
\end{proposition}
\begin{proof}
    Let $\widetilde{M}\subset M$ be the $\A_{g,r}$-module generated by $M_0$. As a submodule of $M$ it is also holonomic with $\GKdim(\widetilde{M}) = (g-r+1)\dim(G)$. By Lemmata~\ref{lem:ns-direct-sum} and \ref{lem:s-direct-sum} proved below, $\widetilde{M}$ is of the form $\OqFRT\otimes M_0$ in the separating resp.\ $\OqG\otimes M_0$ in the non-separating case and thus 
    \[\GKdim(\widetilde{M}) = \GKdim(M_0)+\dim(G)~.\]
    In particular, $M_0$ is holonomic as an $\A'$-module. 
\end{proof}
\subsubsection*{Separating case}

The following identities can be easily checked via using the algebra relations and induction on $N$: 
\begin{align}
    a^2_1 \left(f^1_j\right)^N &= q^{-N}\left(f^1_j\right)^{N} a^2_1 + q^{-1}\numberq{-2N}{q}f^2_j \left(f^1_j\right)^{N-1}a^2_2\label{eq:a21-f1jpowers}\\
    a^1_2 \left(f^2_j\right)^N &= q^{-N}\left(f^2_j\right)^{N} a^1_2 + \numberq{-2N}{q}f^1_j \left(f^2_j\right)^{N-1}a^2_2\label{eq:a12-f2jpowers}
\end{align}
Using these relations, one can show by induction that $f^i_j$ acts without torsion on $M_0$, \ie $(f^i_j)^N m = 0$ implies $m=0$. For instance, if $(f^1_1)^N m = 0$ then acting by $a^2_1$ and using \eqref{eq:a21-f1jpowers} gives $f^2_1(f^1_1)^{N-1}m = 0$. In particular, $\detq(F)(f^1_1)^{N-1}m = f^2_2 (f^1_1)^Nm - qf^1_2f^2_1 (f^1_1)^{N-1}m  = 0$ and the result follows by induction. The same argument applies for $(f^1_2)^N m = 0$ and using \eqref{eq:a12-f2jpowers} for $(f^2_j)^N m =0$. 

\begin{lemma}\label{lem:ns-direct-sum}
    Let $\mathcal{B}$ denote the basis of $\OqFRT$. Then, $\sum_{b\in\mathcal{B}}b M_0$ is a direct sum. In particular, $M_0$ generates the submodule $\OqFRT\otimes M_0$ in $M.$
\end{lemma}
\begin{proof}
    We first show that $\sum_{k+l=N}(f^1_1)^k (f^1_2)^lM_0$ is a direct sum by induction on $N$. Suppose we have 
    \[\sum_{k+l=N}{(f^1_1)^k (f^1_2)^l m_{k}} =0~.\]
    Acting by $a^2_1$ and using \eqref{eq:a21-f1jpowers} we obtain 
    \[\sum_{k+l=N}{\left(q^{-k}\numberq{-2l}{q}f^2_2 \left(f^1_1\right)^{k}\left(f^1_2\right)^{l-1} + \numberq{-2k}{q}f^2_1 \left(f^1_1\right)^{k-1}\left(f^1_2\right)^{l}\right) m_k} = 0\]
    Further acting by $f^1_2$ and using the definition of $\detq(F)$ the above equation becomes 
    \[\sum_{\substack{k+l=N\\ k>0}}{\left(q^{-2k-1}\numberq{-2l}{q}\detq(F) + \numberq{-2N}{q}f^1_2f^2_1 \right) (f^1_1)^{k-1}(f^1_2)^lm_k} = -q^{-1}\numberq{-2N}{q}f^2_2(f^1_2)^{N}m_0~. \]
    Applying the initial assumption one can easily deduce 
    \[\detq(F)\sum_{\substack{k+l=N\\ k>0}}{\numberq{-2k}{q}(f^1_1)^{k-1} (f^1_2)^l m_k} = 0~.\] By induction we have $m_{k} = 0$ for all $k>0$ which implies $(f^1_2)^l m_0 = 0$. Since $f^1_2$ acts without torsion we have shown that $m_k=0$ for all $l$. 

    The same argument applies to show that the sum $\sum_{k+l = N}(f^2_1)^k (f^2_2)^l M_0$ is direct where we instead act by $a_1^2$ and $f^2_1$. This can be used to show that $\sum_{k+l+r+s =N}(f^1_1)^k (f^1_2)^l(f^2_1)^r (f^2_2)^sM_{0}$ is direct, by acting with $a^2_2$ which commutes with $f^1_j$ but $q^2$-commutes with $f^2_j$. Finally, to show that $\sum_{k,l,r,s}M_0^{(k,l,r,s)}$ is a direct sum, one can act by $\detq(A)$ which detects the total degree $k+l+r+s$ and use the above.    
\end{proof}

\subsubsection*{Non-separating case}

The following identities can be checked via induction: 
\begin{align}
    a^2_1 (\partial^2_2)^N &= q^{2N} (\partial^2_2)^N a^2_1 + \numberq{2N}{q}\partial^2_1 (\partial^2_2)^{s-1} a^2_2\label{eq:a21-d22powers}\\ 
    a^1_2 (\partial^2_1)^{N} &= (\partial^2_1)^N + q^{-2}\numberq{2N}{q}\partial^2_2 (\partial^2_1)^{N-1} a^2_2 \label{eq:a12-d21powers}
\end{align}
These relations can be used to show that $\partial^i_j$ acts without torsion on $M_0$. For instance, applying $(a^2_1)^N$ to the relation $(\partial^1_1)^N =0$ implies $(\partial^2_1)^Nm=0$. Thus, it suffices to check that $\partial^1_2,\partial^2_1,\partial^2_2$ act without torsion. For instance, if $(\partial^2_1)^N m=0$ then applying $a^1_2$ and \eqref{eq:a12-d21powers} implies $\partial^2_2 (\partial^2_1)^{N-1} m$. Thus, $\detq(D) (\partial^2_1)^{N-1}m=0$ and the claim follows by induction. The same argument applies if $(\partial^2_2)^N m = 0$ by using $\eqref{eq:a21-d22powers}$.  

\begin{lemma}\label{lem:s-direct-sum}
    Let $\mathcal{B}$ denote the basis of $\OqG$. Then, $\sum_{b\in\mathcal{B}}b M_0$ is a direct sum. In particular, $M_0$ generates the submodule $\OqG\otimes M_0$ in $M.$
\end{lemma}
\begin{proof}
    Similar to the proof of Lemma~\ref{lem:s-direct-sum} we act by either $a^1_2$ or $a^2_1$ to provide induction. For instance, since $\detq(A)$ detects the total degree in $\partial$'s, we only need to prove that $\sum_{k+l+r+s = N}(\partial^1_1)^k(\partial^2_1)^l(\partial^2_2)^r(\partial^1_2)^sM_0$ is direct. To do so, we first use $a^1_2$ to reduce the degree in $\partial^1_1$ and subsequently in $\partial^2_1$. After reducing to $\sum_{r+s=N}(\partial^2_2)^r(\partial^1_2)^sM_0$ we use $a^2_1$ to reduce the degree in $\partial^2_2$ and then $\partial^1_2$. 
\end{proof}

\subsection{The long exact sequence}\label{subsec:les}

Let now $M$ be a holonomic $\A_{g,r}$-module. From the localisation maps $\phi_i: M\rightarrow S^{-1}_iM$ and the results of the previous section we have long exact sequences of holonomic modules 
\begin{equation*}
    0 \rightarrow K_i:= \ker(\phi_i)\rightarrow M \xrightarrow{\phi_i} M_i \rightarrow C_i:= \operatorname{coker}(\phi_i)\rightarrow 0~.
\end{equation*}
The module $K_i$ is the $S_i$-torsion submodule of $M$, \ie 
\[K_i =\{m\in M\mid \exists P\in S_i: P.m =0\}\]
while $C_i$ is also $S_i$-torsion. 

\begin{theorem}\label{thm:preservation}
Let $F: \A_{g,r}\modu\rightarrow \A'\modu$ denote the transfer functor $F=i^\ast$ in the non-separating case resp.\ $F=\delta$ in the separating case and $\A'=\A_{g-1,r}$ resp.\ $\A'=\A_{g_1,r_1}\boxtimes \A_{g_2,r_2}$.  
Let $M$ be a holonomic $\A_{g,r}$-module for $G = \GL_2, \SL_2$. Then, the derived functor $LF( M)$ is an object in $\D^b_\mathrm{hol}(\A')$, \ie $H^k(LF(M))$ are holonomic $\A'$-modules for all $k$.
\end{theorem}
\begin{proof}
    Let $ M$ be a holonomic $\A_{g,r}$-module and consider the exact sequence obtained from the localisation map $\phi_1$:
    \[0\rightarrow K_1 \rightarrow M \xrightarrow{\phi_1} S_1^{(1)}(M) \rightarrow C_1 \rightarrow 0~.\] 
    All modules in the above sequence are holonomic $\A_{g,r}$-modules since $S^{(1)}M$ is holonomic by Lemma \ref{lem:s1} while $K_1$ and $C_1$ are holonomic as submodules resp.\ quotients of holonomic modules, see Remark~\ref{rem:thick-subcat}. 
    Splitting the sequence to 
    \[0\rightarrow K_1 \rightarrow M \rightarrow \operatorname{Im(\phi_1)}\rightarrow 0 \quad\text{and}\quad 0\rightarrow \operatorname{Im}(\phi_1) \rightarrow S^{(1)}_1(M)\rightarrow C_1 \rightarrow 0\]
we reduce (by the left short exact sequence) our claim to: 1) $M$ is $S_1$-torsion and 2) $M$ is $S_1$-torsion-free. The short exact sequence on the right reduces case 2) to case 1) in the following way: If $M$ is $S_1$-torsion-free and thus 
\[0\rightarrow M \rightarrow S^{-1}_1M^{(1)} \rightarrow C_1 \rightarrow 0\]
is exact, then applying $ LF$ and Lemma~\ref{lem:Inv-Im-Loc} we obtain isomorphisms $H^{-k-1}(LF( C_1)) \cong H^{-k}(LF(M))$. 

Henceforth, we can assume that $M$ is $S_1$-torsion and holonomic. Repeating for $S_2,S_3,S_4$ using Lemmata \ref{lem:s2}, \ref{lem:s3} and \ref{lem:s4} we reduce  to the case that $M$ is $S_1,S_2,S_3,S_4$-torsion, \ie 
\[M = \bigcap_{i=1}^{4}t_{S_i}(M)~.\]
In particular, the statement reduces to the holonomicity of 
\[M_0:= \ker(\detq{A}-1)\cap \ker(a^1_2)\cap \ker(a^2_1) \cap \ker(a^2_2 -1)= H^{-\dim(G)}(LF (M)),\] 
see \eqref{eq:Koszul-GL2}, \eqref{eq:Koszul-differentials-GL2} for $G=\GL_2$ and \eqref{eq:Koszul-SL2}, \eqref{eq:Koszul-differentials-SL2} for $\SL_2$. The module $M_0$ was shown to be holonomic as an $\A'$-module in Proposition~\ref{prop:kernel}. 
\end{proof}

\subsection{Proof of main theorem}

Finally, we are able to show for $q$ generic that internal skein modules of $3$-manifolds are holonomic over the associated internal skein algebra.

\begin{theorem}\label{thm:main-thm-skein}
Let $M$ be a 3-manifold with a closed boundary surface $\Sigma$ and suppose that $q$ is generic. Let $\Sigma^\ast$ be a surface obtained from removing one disk on each connected component of $\Sigma$ and fixing a gate on each resulting boundary component of non-zero genus. Then the internal skein module 
	\begin{equation*}
		\intskmod_{q,\GL_2}(M)
	\end{equation*} 
	is holonomic over the internal skein algebra $\A_{\Sigma^\ast}$. 
\end{theorem}
\begin{proof}
Fix a Heegaard splitting of $M$
\begin{equation*}
    M \cong H_\Sigma^\gamma \circ C_{\Sigma, \partial M}
\end{equation*}
where $C_{\Sigma, \partial M}$ is a compression body obtained by iteratively gluing $2$-handles along multicurves $\alpha$ living in the $\Sigma \times \{1\}$ boundary component of the cylinder $\Sigma\times I$. By Proposition Proposition~\ref{prop:compression} the compression body $C_{\Sigma,\partial M}$ further decomposes to a composition 
\[
C_{\Sigma,\partial M} = M_{f_0}\circ C_1 \circ M_{f_1}\circ \cdots\circ M_{f_{m-1}} \circ C_m \circ X
\]
where $C_i$'s are compression bodies obtained from a 2-handle attachment along a standard curve, $M_{f_i}$ are mapping cylinders and $X$ corresponds to gluing 3-handles. 

In particular, the internal skein module of $M$ decomposes to: 
\[\intskmod_{q,G}(M) \cong \mathcal{T}_{C_m}\otimes_{\A_{\Sigma_{m-1}}}^{f_{m-1}}\cdots \otimes^{f_1}_{\A_{\Sigma_1}}\mathcal{T}_{C_1} \otimes_{\A_{\Sigma}}^{\gamma}\intskmod_{q,G}(H_\Sigma)\]
where we use $\otimes_{\A_{\Sigma_i}}^{f_i}$ to denote the relative tensor product over $\A_{\Sigma_i}$ where the left $\A_{\Sigma_i}$-module is twisted via the mapping class $f_i$.

Twisting a holonomic $\A_\Sigma$-module $M$ by a mapping class $f\in \Mod_{\Sigma}$ yields again a holonomic $\A_\Sigma$-module ${}^f M$. This is because twisting by an algebra automorphism preserves finite generation and the Lagrangian condition as it induces a symplectomorphism. Combining this with Theorem~\ref{thm:preservation} and that $\intskmod_{q,G}(H_\Sigma)$ is holonomic we conclude the proof. 
\end{proof}

\begin{corollary}
	For $q$ generic, the ordinary skein module $\skmod_{q,G}(M)$ is finitely generated over the ordinary skein algebra $\skalg_{q,G}(\Sigma)$. 
\end{corollary}
\begin{proof}
    By Theorem~\ref{thm:main-thm-skein} there exists an $\A_{\Sigma^\ast}$-epimorphism in $\Vect$
    \begin{equation*}
        \A_{\Sigma^\ast}^{\oplus n} \twoheadrightarrow\intskmod_{q,G}(M)~
    \end{equation*}
    which lifts to an $\A_{\Sigma\ast}$-epimorphism in $\Ahat=\Rep_qG$ as the forgetful functor
    \[\A_{\Sigma^\ast}\modu(\Rep_qG)^{\mathrm{str}}\rightarrow \A_{\Sigma^\ast}\modu(\Vect) \]
    is fully faithful for $q $ generic.
    However, since 
    which by evaluation on $\unit \in \mathcal{A}$ induces a $\skalg_{q,\GL_2}(\Sigma^\ast)$-module epimorphism 
    \begin{equation*}
        \skalg_{q,G}(\Sigma^\ast)^{\oplus n} \rightarrow \skmod_{q,G}(M)~.
    \end{equation*}
    Finally, the above epimorphism factors through $\skalg_{q,G}(\Sigma)^{\oplus n}\rightarrow \skmod_{q,G}(M)$ by the canonical algebra morphism 
    \begin{equation*}
        \skalg_{q,G}(\Sigma^\ast) \rightarrow \skalg_{q,G}(\Sigma)~.
    \end{equation*}
\end{proof}

\newcommand{\arxiv}[2]{\href{http://arXiv.org/abs/#1}{#2}}
\newcommand{\doi}[2]{\href{http://doi.org/#1}{#2}}

\end{document}